\documentclass[11pt]{amsart}

\usepackage{amsmath,amsfonts,amssymb,amsthm}
\usepackage{txfonts}
\usepackage{latexsym,bm,graphicx}
\usepackage{mathrsfs}
\usepackage{color}

\title[Local Li-Yau's Estimates]{Local Li-Yau's estimates on  $RCD^*(K,N)$ metric measure spaces}
\author{Hui-Chun Zhang}
\address{Department of Mathematics\\  Sun Yat-sen University\\ Guangzhou 510275\\ E-mail address: zhanghc3@mail.sysu.edu.cn}
\author{Xi-Ping Zhu}
\address{Department of Mathematics\\  Sun Yat-sen University\\ Guangzhou 510275\\ E-mail address: stszxp@mail.sysu.edu.cn}

\newtheorem{thm}{Theorem}[section]
\newtheorem{prop}[thm]{Proposition}
\newtheorem{lem}[thm]{Lemma}

\newtheorem{cor}[thm]{Corollary}

\theoremstyle{definition}
\theoremstyle{remark}

\newtheorem{defn}[thm]{Definition}
\newtheorem{rem}[thm]{Remark}

\numberwithin{equation}{section}

\newcommand{\ls}{\leqslant}
\newcommand{\gs}{\geqslant}

\newcommand{\ip}[2]{\langle{#1},{#2}\rangle}

\newcommand{\R}{\mathbb{R}}

\newcommand{\La}{\mathscr{L}}
\newcommand{\E}{{\mathscr E}}

\begin{document}
%\today

\begin{abstract}In this paper, we will study the (linear) geometric analysis on metric measure spaces.
We will establish a local Li-Yau's estimate for weak solutions of the heat equation and prove  a sharp
Yau's gradient gradient for harmonic functions on metric measure spaces, under the Riemannian curvature-dimension condition $RCD^*(K,N)$.
\end{abstract}

\maketitle
\tableofcontents
\setcounter{tocdepth}{1}

\section{Introduction}

 In the field of geometric analysis, one of the fundamental results is the following Li-Yau's local gradient estimate for solutions
  of the heat equation on a complete Riemannian manifold.
 \begin{thm}[Li-Yau  \cite{ly86}]\label{thm-ly} Let $(M^n,g)$ be an $n$-dimensional complete Riemannian manifold, and
  let $B_{2R} $ be a geodesic ball of  radius $2R$  centered at  $O\in M^n$. Assume that $Ric(M^n)\gs-k$ with $k\gs0$.
  If $u(x,t)$ is a smooth positive solution of the heat equation $\Delta u=\partial_tu$ on $B_{2R}\times(0,\infty)$, then for any $\alpha>1$,
  we have  the following gradient estimate in $B_R$:
 \begin{equation}\label{eq-ly}
 \sup_{x\in B_R}\big( |\nabla f|^2  -\alpha\cdot  \partial_t f \big)(x,t)\ls \frac{C\alpha^2}{R^2}\Big(\frac{\alpha^2}{\alpha^2-1}+\sqrt kR\Big)+\frac{n\alpha^2k}{2(\alpha-1)}+\frac{n\alpha^2}{2t}
 \end{equation}
where $f:=\ln u$ and $C$ is a constant depending only on $n$.
 \end{thm}
By letting $R\to\infty$ in (\ref{eq-ly}), one gets a global gradient estimate, for any $\alpha>1$, that
\begin{equation}\label{eq1.2}
 |\nabla f|^2 -\alpha\cdot \partial_t f\ls   \frac{n\alpha^2k}{2(\alpha-1)}+\frac{n\alpha^2}{2t}.
 \end{equation}

There is a rich literature on extensions and improvements of the Li-Yau inequality, both the local version (\ref{eq-ly})
 and the global version (\ref{eq1.2}),
 to diverse settings and evolution equations,
for example, in the setting of Riemannian manifolds with Ricci curvature bounded below \cite{dav89,bl06,sz06,lx11,lee12}, in the setting of  weighted Riemannian manifolds with Bakry-Emery Ricci curvature bounded below  \cite{bq99,li05,qia14,bbg14} and some non-smooth setting  \cite{bhllmy15,qzz13},  and so on.

Let $(X,d,\mu)$ be a complete, proper metric measure space  with ${\rm supp}(\mu)=X.$
The  \emph{curature-dimension condition} on  $(X,d,\mu)$  has been introduced by Sturm \cite{stu06}
 and Lott-Villani  \cite{lv09}. Given $K\in \R $ and $N\in[1,\infty]$, the curvature-dimension condition $CD(K,N)$ is
  a synthetic notion  for ``generalized Ricci curvature   $\gs K$ and dimension $\ls N$" on  $(X,d,\mu)$.
  Bacher-Sturm \cite{bs10} introduced the \emph{reduced} curvature-dimension condition $CD^*(K,N)$, which
  satisfies a local-to-global property. On the other hand, to  rule out Finsler geometry, Ambrosio-Gigli-Savar\'e \cite{ags15} introduced
   the  \emph{Riemannian} curvature-dimension condition $RCD(K,\infty)$, which assumes that the heat flow on $L^2(X)$ is linear.
Remarkably,  Erbar-Kuwada-Sturm \cite{eks15} and Ambrosio-Mondino-Savar\'e \cite{ams15} introduced a dimensional
 version of Riemannian curvature-dimension condition $RCD^*(K,N)$  and proved that it  is equivalent to a Bakry-Emery's Bochner inequality via an abstract $\Gamma_2$-calculus for semigroups. In the case of Riemannian geometry, the notion  $RCD^*(K,N)$ coincides with the original  Ricci curvature $\gs K$ and dimension $\ls N$, and for the case of  the weighted manifolds $(M^n,g,e^\phi\cdot {\rm vol}_g)$, the notion $RCD^*(K,N)$ coincides with the corresponding  Bakry-Emery's curvature-dimension condition (\cite{stu06,lv09}). In the setting of  Alexandrov geometry,
  it is implied by generalized (sectional) curvature bounded below in the sense of Alexandrov \cite{pet11,zz10}.

Based on the  $\Gamma_2$-calculus for the heat flow $(H_tf)_{t\gs0}$  on $L^2(X)$, many important results
 in geometric analysis have been obtained on a  metric measure space $(X,d,\mu)$ satisfying $RCD^*(K,N)$ condition.
 For instance, Li-Yau-Hamilton estimates for the heat flow $(H_tf)_{t\gs0}$ \cite{gm14,jia15,jiang-z16} and
   spectral gaps \cite{lv07,qzz13,jz16} for the infinitesimal generator of $(H_tf)_{t\gs0}$.

In this paper, we will study the \emph{locally} weak solutions of the heat equation on a metric measure space $(X,d,\mu)$.
Let $\Omega\subset X$  be  an open set. The $RCD^*(K,N)$ condition implies that the Sobolev space $W^{1,2}(\Omega)$ is a Hilbert space.
Given an interval $I\subset\mathbb R$, a function $u(x,t)\in W^{1,2}(\Omega\times I)$ is called a \emph{locally weak solution}
 for the heat equation on $\Omega\times I$ if it  satisfies
\begin{equation}\label{eq1.3}
-\int_I\int_{\Omega}\!\ip{\nabla u}{\nabla \phi}d\mu dt = \int_I\int_{\Omega}\!\frac{\partial u}{\partial t}\cdot\phi d\mu dt
\end{equation}
for all Lipschitz functions $\phi$ with compact support in $\Omega\times I$, where the inner product $\ip{\nabla u}{\nabla \phi}$ is given
 by polarization in $W^{1,2}(\Omega)$.

Notice that the  locally  weak solutions $u(x,t)$ do not form a semi-group in general.  The method of $\Gamma_2$-calculus for the heat flow
 in the previous works \cite{gm14,jia15,jz16} is no longer be suitable for the problems on locally weak solutions of the heat equation.

To seek an  appropriate method to deal with the locally weak solutions for the heat equation,  let us recall  what is the proof of Theorem \ref{thm-ly} in the smooth context.  There are two main ingredients:
 the Bochner formula and a  maximum principle.
 The Bochner formula states that
\begin{equation}\label{bochner-formula}
\frac{1}{2}\Delta |\nabla f|^2\gs \frac{(\Delta f)^2}{n}+\ip{\nabla f}{\nabla \Delta f}+K|\nabla f|^2
\end{equation}
for any $C^3$-function $f$ on $M^n$  with Ricci curvature  $Ric(M^n)\gs K$ for some $K\in \mathbb R$. The maximum principle  states that if $f(x)$ is of $C^2$ on $M^n$ and if it achieves its a local maximal value at point $x_0\in M^n$, then we have
\begin{equation}\label{maximum}
\nabla f(x_0)=0\qquad {\rm and }\qquad \Delta f(x_0)\ls 0.
\end{equation}
For simplification, we only consider the special case  that  $u(x,t)$ is a smooth positive solution for heat equation on  a compact manifold  $M^n$ with  $Ric(M^n)\gs0$.  By using the Bochner formula to $\ln u$, one  deduces a differential inequality
$$\Big(\Delta -\frac{\partial}{\partial t}\Big) F\gs -2\ip{\nabla f}{\nabla F}+\frac{2}{nt}F^2-\frac{F}{t},$$
where $f=\ln u$ and $F=t\big(|\nabla f|^2-\partial_tf\big).$ Then by  using the maximum principle to  $F$  at one of its maximum points $(x_0,t_0)$, one gets the desired Li-Yau's estimate
$$  \max F=F(x_0,t_0)\ls\frac{n}{2}.$$

In this paper,  we   want to extend these two main ingredients to non-smooth metric measure spaces.
Firstly, let us consider the Bochner formula in non-smooth context.  Let $(X,d,\mu)$ be a metric measure space with $RCD^*(K,N)$. Erbar-Kuwada-Sturm \cite{eks15} and Ambrosio-Mondino-Savar\'e \cite{ams15}  proved that  $RCD^*(K,N)$ condition is  equivalent to a Bakry-Emery's Bochner inequality for the heat flow $(H_tf)_{t\gs0}$ on $X$. This provides a global version of Bochner formula for the infinitesimal generator of the heat flow $(H_tf)_{t\gs0}$ (see Lemma \ref{lem-boch}). On the other hand, a good cut-off function has been obtained in \cite{ams15,mn14,hkx13}. By combining these two facts and an argument in \cite{hkx13}, one can localize the global version of Bochner formula in \cite{eks15,ams15} to  a local one.

To state the local version of Bochner formula, it is more  convenient to work with a notion of
 the weak Laplacian, which is a slight modification from \cite{gig15,gm13}.  Let $\Omega\subset X$ be an open set. Denote
 by  $H^1(\Omega):=W^{1,2}(\Omega)$ and  $H^1_0(\Omega):=W^{1,2}_0(\Omega)$.
The \emph{weak Laplacian} on $\Omega$ is an operator $\mathscr L$ on $H^1(\Omega)$ defined by:  for each  function $f\in H^1(\Omega)$, $\mathscr Lf$ is a functional acting on $H^1_0(\Omega)\cap L^\infty(\Omega)$ given by
$$\mathscr Lf(\phi):=-\int_\Omega\ip{\nabla f}{\nabla \phi}d\mu\qquad \forall\ \phi\in H^1_0(\Omega)\cap L^\infty(\Omega).$$
In the case when it holds
\begin{equation*}
\mathscr Lf(\phi)\gs \int_\Omega h\cdot\phi d\mu\qquad \forall\ 0\ls \phi\in H^1_0(\Omega)\cap L^\infty(\Omega)
\end{equation*}
 for some function $h\in L^1_{\rm loc}(\Omega)$, then it is well-known \cite{hom89} that the weak Laplacian $\La f$ can be extended to a signed Radon measure on $\Omega$. In this case, we denote by
 $$\La f\gs h\cdot\mu$$
 on $\Omega$ in the sense of distributions.

Now, the local version of Bochner formula is given as follows.
 \begin{thm}[\cite{ams15,hkx13}]\label{thm-bochner}
 Let $(X,d,\mu)$ be a metric measure space with $RCD^*(K,N)$ for some $K\in \mathbb R$ and $N\gs1$. Assume that $f\in H^{1}(B_R)$ such that $\mathscr Lf$ is a signed measure on $B_R$  with the density $g\in H^1(B_R)\cap L^\infty(B_R)$.
Then we have $|\nabla f|^2\in H^{1}(B_{R/2}) \cap L^\infty(B_{R/2})$  and that  $\La(|\nabla  f |^2)$ is a signed Radon measure  on
 $B_{R/2}$ such that
\begin{equation*}
 \frac{1}{2}\La(|\nabla  f |^2 ) \gs  \Big[\frac{g^2}{N}+\ip{\nabla f}{\nabla  g}+K|\nabla f|^2\Big]\cdot\mu
 \end{equation*}
on  $B_{R/2}$ in the sense of distributions.
\end{thm}

Next, we consider to extend the maximum principle (\ref{maximum}) from smooth Riemannian manifolds to non-smooth metric measure spaces $(X,d,\mu)$.
A simple observation is that  the maximum principle (\ref{maximum}) on a smooth manifold $M^n$ has  the following equivalent form:

\emph{Suppose that $f(x)$ is of $C^2$ on $M^n$ and that it achieves its a local maximal value at point $x_0\in M^n$. Given any $w\in C^1(U)$ for some neighborhood  $U$ of $x_0$. Then we have}
  $$\Delta f(x_0)+\ip{\nabla f}{\nabla w}(x_0)\ls0.$$
In the following result, we will extend the observation to  the non-smooth context.  Technically, it is  our main effort in the paper.
 \begin{thm}\label{thm1.3}
Let $\Omega$ be a bounded domain in a metric measure space $(X,d,\mu)$ with $RCD^*(K,N)$ for some $K\in\mathbb R$ and $N\gs1.$
Let $f(x)\in H^{1}(\Omega)\cap L^\infty_{\rm loc}(\Omega)$  such that  $\La f $ is a signed Radon measure with $\La^{\rm sing}f\gs 0$, where $\La^{\rm sing}f$ is the singular part with respect to $\mu$.
Suppose that $f$ achieves one of its   strict maximum in  $\Omega$ in the sense that: there exists a neighborhood $U\subset\subset \Omega$ such that
\begin{equation*}
\sup_{U}f>\sup_{\Omega\backslash U}f.
\end{equation*}
Then, given any $w\in  H^{1}(\Omega)\cap L^\infty(\Omega)$, there exists a sequence of points $\{x_j\}_{j\in\mathbb N}\subset U$  such that they are the approximate continuity points of $\La^{\rm ac}f$ and $\ip{\nabla f}{\nabla w}$, and that
$$f(x_j)\gs \sup_\Omega f-1/j\qquad{\rm and }\qquad  \La^{\rm ac}f(x_j)+\ip{\nabla f}{\nabla w}(x_j)\ls 1/j.$$
Here and in the sequel of this paper, $\sup_{U}f$ means  ${\rm ess}\sup_{U}f$.
\end{thm}

This result is close to the spirit of  the Omori-Yau maximum principle \cite{omo67,y75}.  It has also some similarity with the approximate versions of the maximum principle developed, for instance by Jensen \cite{jen88}, in the theory of second order viscosity solutions.

A similar parabolic version of the maximum principle, Theorem \ref{max-parabolic}, will be given in \S4.

After obtaining the above Bochner formula and  the maximum principle (Theorem \ref{thm-bochner} and Theorem \ref{max-parabolic}), we will show the following  Li-Yau type gradient estimates for locally weak solutions of the heat equation, which is our main purpose in this paper.
\begin{thm}\label{thm1.4}
 Let $K\gs 0$ and $N\in[1,\infty)$, and let $(X,d,\mu)$ be a metric measure space satisfying  $RCD^*(-K,N)$.
Let $T_*\in(0,\infty]$ and let  $B_{2R}$ be a geodesic ball of radius $2R$ centered at $p\in X$, and let  $u(x,t)\in W^{1,2}\big(B_{2R}\times(0,T_*)\big)$ be a positive locally weak solution of the heat equation on $B_{2R}\times(0,T_*)$. Then,  given any $T\in(0,T_*)$, we have the following local gradient estimate
 \begin{equation}\label{eq1.6}
 \begin{split}
 \sup_{B_R\times(\beta\cdot T,T]}\Big(|\nabla f|^2-\alpha\cdot \frac{\partial}{\partial t}f\Big)(x,t) \ls& \max\bigg\{1, \frac{1}{2}+\frac{K T}{2(\alpha-1)}\bigg\} \cdot \frac{N\alpha^2}{2 T}\cdot  \frac{1}{\beta^2}\\
 &\ \ +\frac{C_{N}\cdot \alpha^4 }{R^2 (\alpha-1)} \cdot  \frac{1}{(1-\beta)\beta^2}+ C_{N}\cdot \frac{\alpha^2}{\beta^2}\cdot  \Big(\frac{\sqrt K}{R}+\frac{1}{R^2}\Big)
\end{split}
 \end{equation}
for any $\alpha>1$ and any $\beta\in(0,1)$, where $f=\ln u$, and $C_{N}$ is a constant depending only on $N$. Here and in the sequel of this paper, $\sup_{B_R\times[a,b]}g$ means  ${\rm ess}\sup_{B_R\times[a,b]}g$ for a function $g(x,t)$.
 \end{thm}

The local boundedness and the Harnack inequality for locally weak solutions of the heat equation have been established by  Sturm \cite{stu95,stu96} in the setting of abstract local Dirichlet form and  by Marola-Masson  \cite{mm13} in the setting of metric measure with a standard volume doubling property and supporting a $L^2$-Poincare inequality. Of course, they are available on  metric measure spaces $(X,d,\mu)$ satisfying $RCD^*(K,N)$  for some $K\in\mathbb R$ and $N\in [1,\infty)$. In particular, any locally weak solutions for the heat equation  must be locally H\"older continuous.

As a consequence of Theorem \ref{thm1.4}, letting $R\to \infty$ and $\beta\to 1$, we get the following global gradient estimates.

\begin{cor}\label{cor1.5}
 Let $(X,d,\mu)$ and $K, N,T_*$ be as in the Theorem \ref{thm1.4}.  Let  $u(x,t)$ is a positive solution of the heat equation on $X\times(0,T_*)$.
Then, for almost all $T\in  (0,T_*)$,  the following  gradient estimate holds
 \begin{equation*}
 \sup_{x\in X}\Big(|\nabla f|^2-\alpha\cdot \frac{\partial}{\partial t}f\Big)(x,T) \ls \max\bigg\{1, \frac{1}{2}+\frac{K T}{2(\alpha-1)}\bigg\} \cdot \frac{N\alpha^2}{2 T}\ls \Big(1+\frac{K T}{2(\alpha-1)}
\Big) \cdot \frac{N\alpha^2}{2T}
 \end{equation*}
 for any $\alpha>1,$ where $f=\ln u.$
 \end{cor}

As  another application of the maximum principle, Theorem \ref{thm1.3},  and the Bochner formula, we will deduce a \emph{sharp} Yau's gradient estimate for harmonic functions on metric measure spaces satisfying  $RCD^*(-K,N)$ for $K\gs0$ and $N>1$.

Let us recall the classical local Yau's gradient estimate in geometric analysis (see \cite{cy75,y75,lw02}).
Let $M^n$ be an $n(\gs2)$-dimensional complete non-compact Riemannian manifold with  $Ric(M^n)\gs-k$ for some $k\gs0$. The local Yau's gradient estimate asserts that for any positive harmonic function $u$ on $B_{2R}$, then
\begin{equation}\label{eq1.7}
\sup_{B_{R}}|\nabla \ln u|\ls \sqrt{(n-1)k}+\frac{C(n)}{R}.
\end{equation}
 In particular, if $u$ is positive harmonic on $M^n$ and $Ric\gs-(n-1)$ on $M^n$ then it follows that $|\nabla \log u|\ls n-1$  on $M^n$. This result is sharp, in fact the equality case was characterized in \cite{lw02}. This means that for $k = n-1$ in (\ref{eq1.7}) the factor $\sqrt{n-1}$ on the right hand side is sharp.

Let   $(X,d,\mu)$  be a metric measure space satisfying $RCD^*(-K,N)$ for some $K\gs0$ and $N\in(1,\infty)$.
It was proved in \cite{jia14} the following form of Yau's gradient estimate that, for any positive harmonic function $u$ on $B_{2R}\subset X$, it holds
\begin{equation}\label{eq1.10}
\sup_{B_{R}}|\nabla \ln u|\ls C(N,K,R).
\end{equation}
In the setting of Alexandrov spaces, by using a Bochner formula and an argument of Nash-Moser iteration,  it was  proved in \cite{zz12,hx14} the following form of Yau's gradient estimate holds: given an $n$-dimensional Alexandrov space $M$ and a positive harmonic function $u$ on $B_{2R}\subset M$, if the generalized Ricci curvature on $B_{2R}\subset M$ has a lower bound $Ric\gs-k$,  $k\gs0$, in the sense of \cite{zz10}, then
\begin{equation*}
\sup_{B_{R}}|\nabla \ln u|\ls C_1(n)\sqrt k+\frac{C_2(n)}{R}.
\end{equation*}
 Indeed, by applying Theorem \ref{thm-bochner}, the same argument in \cite{zz12,hx14} implies this estimate still holds for harmonic function $u$ on a metric measure  space $(X,d,\mu)$ with $RCD^*(-k,n).$ However, it seems hopeless to improve the fact $C_1(n)$ to the sharp $\sqrt{n-1}$ in (\ref{eq1.7}) via a Nash-Moser iteration   argument.

The last result in this paper is to establish a sharp local Yau's gradient estimate on metric measure spaces with Riemannian curvature-dimension condition.
\begin{thm}\label{thm1.6}
 Let $K\gs 0$ and $N\in (1,\infty)$, and let $(X,d,\mu)$ be a metric measure space satisfying  $RCD^*(-K,N)$.
Let   $B_{2R}$ be a geodesic ball of radius $2R$ centered at $p\in X$, and let $u(x)$ be a positive harmonic function on $B_{2R}$. Then the following local Yau's gradient estimate holds
 \begin{equation}\label{eq1.9}
 \sup_{B_{R}}|\nabla \ln u|\ls   \sqrt{ \frac{1+\beta}{1-\beta}\cdot (N-1) K }+\frac{C(N)}{\sqrt{\beta(1-\beta) }\cdot R}
 \end{equation}
 for any $\beta\in(0,1).$ \end{thm}

\noindent\textbf{Acknowledgements.} H. C. Zhang is partially supported by NSFC 11571374. X. P. Zhu is partially supported by NSFC 11521101.

\section{Preliminaries}

Let $(X,d)$ be a complete metric space and $\mu$ be a Radon measure on $X$ with ${\rm supp}(\mu)=X.$ Denote by $B_r(x)$ the open ball centered at $x$ and radius $r$.
Throughout the paper, we assume that
$X$ is proper (i.e.,  closed balls of finite radius are compact). Denote by $L^p(\Omega):=L^p(\Omega,\mu)$ for any open set $\Omega\subset X$ and any $p\in[1,\infty].$
\subsection{Reduced and Riemannian curvature-dimension conditions}$\ $

Let $\mathscr P_2(X,d)$ be the $L^2$-Wasserstein space over $(X,d)$, i.e., the set of all Borel probability measures $\nu$ satisfying
$$\int_Xd^2(x_0,x)d\nu(x)<\infty$$
for some (hence for all) $x_0\in X$. Given two measures $\nu_1,\nu_2\in \mathscr P_2(X,d)$, the $L^2$-Wasserstein distance between them is given by
$$ W^2(\nu_0,\nu_1):=\inf\int_{X\times X}d^2(x,y)dq(x,y)$$
where the infimum is taken over all couplings $q$ of $\nu_1$ and $\nu_2$, i.e., Borel probability measures $q$ on $X\times X$ with marginals $\nu_0$ and $\nu_1.$
Such a coupling $q$ realizes the $L^2$-Wasserstein distance is called an \emph{optimal coupling} of $\nu_0$ and $\nu_1.$
Let $\mathscr P_2(X,d,\mu)\subset  \mathscr P_2(X,d)$ be the subspace of all measures absolutely continuous w.r.t. $\mu.$ Denote  by $\mathscr P_\infty(X,d,\mu)\subset  \mathscr P_2(X,d,\mu)$ the set of measures in  $\mathscr P_2(X,d,\mu)$ with bounded support.
\begin{defn}
Let $K\in\mathbb R$ and $N\in[1,\infty)$. A metric measure space $(X,d,\mu)$ is called to satisfy the \emph{reduced curvature-dimension condition} $CD^*(K,N)$ if any only if for each pair $\nu_0=\rho_0\cdot\mu,\nu_1=\rho_1\cdot \mu\in \mathscr P_\infty(X,d,\mu)$ there exist an optimal coupling $q$ of them and a geodesic $(\nu_t:=\rho_t\cdot\mu)_{t\in[0,1]}$ in $\mathscr  P_\infty(X,d,\mu)$ connecting them such that for all $t\in[0,1]$ and all $N'\gs N$:
\begin{equation*}
\begin{split}
\int_X\!\rho_t^{-1/N'}d\nu_t\gs \int_{X\times X}\!\Big[\sigma^{(1-t)}_{K/N'}\big(d(x_0,x_1)\big)\rho_0^{-1/N'}(x_0)+ \sigma^{(t)}_{K/N'}\big(d(x_0,x_1)\big)\rho_1^{-1/N'}(x_1)\Big]dq(x_0,x_1),
\end{split}
\end{equation*}
where the function
\begin{equation*}
\sigma^{(t)}_k(\theta):=
\begin{cases}
\frac{\sin(\sqrt k\cdot t\theta)}{\sin(\sqrt k\cdot \theta)},& \quad 0<k\theta^2<\pi^2,\\
t, &\quad  k\theta^2=0,\\
\frac{\sinh(\sqrt{-k}\cdot t\theta)}{\sinh(\sqrt{- k}\cdot \theta)},& \quad k\theta^2<0,\\
\infty,&\quad  k\theta^2\gs\pi^2.
\end{cases}
\end{equation*}
\end{defn}

Given a function $f\in C(X)$, the \emph{pointwise Lipschitz constant} (\cite{che99}) of $f$ at $x$ is defined by
\begin{equation*}
{\rm Lip}f(x):=\limsup_{y\to x}\frac{|f(y)-f(x)|}{d(x,y)}=\limsup_{r\to0}\sup_{d(x,y)\ls r}\frac{|f(y)-f(x)|}{r},
\end{equation*}
where we put ${\rm Lip}f(x)=0$ if $x$ is isolated. Clearly, ${\rm Lip}f$ is a $\mu$-measurable function on $X.$
The \emph{Cheeger energy}, denoted by ${\rm Ch}:\ L^2(X)\to[0,\infty]$, is defined  (\cite{ags14}) by
$${\rm Ch}(f):=\inf\Big\{\liminf_{j\to\infty}\frac 1 2\int_X({\rm Lip}f_j )^2d\mu\Big\},$$
where the infimum is taken over all sequences of Lipschitz functions $(f_j)_{j\in\mathbb N}$ converging to $f$ in $L^2(X).$ In general, ${\rm Ch}$ is a convex and lower semi-continuous functional on $L^2(X)$.
\begin{defn}
A metric measure space $(X,d,\mu)$ is called \emph{infinitesimally Hilbertian} if the associated Cheeger energy is quadratic. Moreover,  $(X,d,\mu)$ is said to satisfy \emph{Riemannian curvature-dimension condition} $RCD^*(K,N)$, for  $K\in\mathbb R$ and $N\in[1,\infty)$,  if it is infinitesimally Hilbertian and satisfies the $CD^*(K,N)$ condition.
\end{defn}

Let $(X,d,\mu)$ be a  metric measure space  satisfying   $RCD^*(K,N)$. For each $f\in D(\rm Ch)$, i.e., $f\in L^2(X)$ and ${\rm Ch}(f)<\infty$, it has
$${\rm Ch}(f)=\frac{1}{2}\int_X|\nabla f|^2d\mu,$$
where $|\nabla f|$ is the so-called minimal relaxed gradient of $f$ (see \S4 in \cite{ags14}). It was proved, according to  \cite[Lemma 4.3]{ags14} and Mazur's lemma, that Lipschitz functions are
 dense in $ D(\rm Ch)$, i.e., for each   $f\in D(\rm Ch)$, there exist  a sequence of Lipschitz functions $(f_j)_{j\in\mathbb N}$ such that
 $f_j\to f$ in $L^2(X)$ and $|\nabla (f_j- f)|\to 0$ in $L^2(X)$.
Since the Cheeger energy ${\rm Ch}$ is a quadratic form, the minimal relaxed gradients bring an inner product as following: given $f,g\in D(\rm Ch)$, it was proved \cite{gig15} that the limit
$$\ip{\nabla f}{\nabla g}:=\lim_{\epsilon\to0}\frac{|\nabla (f+\epsilon\cdot g)|^2-|\nabla f|^2}{2\epsilon}$$
exists in $L^1(X).$ The inner product is bi-linear and satisfies Cauchy-Schwarz inequality, Chain rule and Leibniz
rule (see Gigli \cite{gig15}).

\subsection{Canonical Dirichlet form and a global version of Bochner formula}$\ $

Given an infinitesimally Hilbertian metric measure space $(X,d,\mu)$, the energy $\mathscr E:=2{\rm Ch}$ gives a canonical Dirichlet form on $L^2(X)$ with the domain $\mathbb{V}:=D({\rm Ch})$.
Let $K\in\mathbb R$ and $N\in[1,\infty)$, and let $(X,d,\mu)$ be a  metric measure space  satisfying   $RCD^*(K,N)$. It has been shown \cite{ags15,ags-duke} that
 the canonical Dirichlet form $(\mathscr E,\mathbb{V})$  is strongly local and admits a Carr\'e du champ $\Gamma$ with $\Gamma(f)=   |\nabla f |^2$ of $f\in \mathbb{V}$. Namely, the energy measure of $f\in \mathbb{V}$ is absolutely continuous w.r.t. $\mu$ with the density $|\nabla f |^2$.
Moreover,  the intrinsic distance $d_{\mathscr E}$ induced by $(\mathscr E,\mathbb{V})$ coincides with the original distance $d$ on $X$.

It is worth noticing that if a metric measure space $(X,d,\mu)$ satisfying   $RCD^*(K,N)$ then its associated Dirichlet form $(\mathscr E,\mathbb{V})$ satisfies the standard assumptions: the local  volume doubling property and supporting a local  $L^2$-Poincare inequality (see \cite{stu06,raj12-cvpde}).

Let $\big(\Delta_{\mathscr E},D(\Delta_{\mathscr E})\big)$ and $(H_tf)_{t\gs0}$  denote the infinitesimal generator and the heat flow induced from $(\mathscr E,\mathbb{V})$. Let us recall the Bochner formula (also called the \emph{ Bakry-Emery condition}) in \cite{eks15} as following.
\begin{lem}\label{lem-boch}
Let $(X,d,\mu)$ be a  metric measure space  satisfying   $RCD^*(K,N)$ for $K\in\mathbb R$ and $N\in[1,\infty)$, and let  $(\mathscr E,\mathbb{V})$ be  the associated  canonical Dirichlet form. Then  the following properties hold.

$(i)$ $($\cite[Theorem 4.8]{eks15}$)$\indent If
$f\in D(\Delta_{\mathscr E})$ with $\Delta_{\mathscr E}f\in \mathbb{V}$ and if $\phi\in D(\Delta_{\mathscr E})\cap L^\infty(X)$ with $\phi\gs0$ and $\Delta_{\mathscr E}\phi\in L^\infty(X)$, then we have the Bochner formula:
\begin{equation}\label{bochner}
\begin{split}
\frac{1}{2}\int_X\Delta_{\mathscr E}\phi\cdot|\nabla f|^2d\mu\gs \frac{1}{N}\int_X\phi(\Delta_{\mathscr E} f)^2d\mu+\int_X\phi\ip{\nabla(\Delta_{\mathscr E}f)}{\nabla f}d\mu+K\int_X\phi|\nabla f|^2d\mu.
\end{split}
\end{equation}

$(ii)$ $($\cite[Theorem 5.5]{ams15}$)$\indent If
$f\in D(\Delta_{\mathscr E})$ with $\Delta_{\mathscr E}f\in L^4(X)\cap L^2(X)$ and if $\phi\in  \mathbb{V}$ with $\phi\gs0$, then we have
$|\nabla f|^2\in \mathbb{V}$ and the modified Bochner formula:
\begin{equation}\label{modi-bochner}
\begin{split}
 \int_X\Big(-\frac{1}{2} \ip{\nabla |\nabla f|^2}{\nabla \phi}+&\Delta_{\mathscr E}f\cdot \ip{\nabla f}{\nabla \phi}+\phi\cdot(\Delta_{\mathscr E} f)^2\Big)d\mu\\
  &\gs\int_X \Big(K |\nabla f|^2 + \frac{1}{N} (\Delta_{\mathscr E} f)^2\Big)\cdot\phi d\mu.\qquad
\end{split}
\end{equation}
\end{lem}

We need the following result on the existence of good cut-off functions on $RCD^*(K,N)$-spaces from \cite[Lemma 3.1]{mn14}; see also \cite{gmo15,ams15,hkx13}.
\begin{lem}\label{lem2.4}
Let $(X,d,\mu)$ be a  metric measure space  satisfying   $RCD^*(K,N)$ for $K\in\mathbb R$ and $N\in[1,\infty)$. Then for every $x_0\in X$ and $R>0$ there exists a Lipschitz cut-off function $\chi:  X\to [0,1]$ satisfying:\\
(i) $\ \chi=1$ on $B_{2R/3}(x_0)$ and ${\rm supp}(\chi)\subset B_{R}(x_0)$;\\
(ii) $\ \chi\in D(\Delta_{\E})$ and $\Delta_{\E}\chi\in \mathbb V\cap L^\infty(X)$, moreover $|\Delta_{\E}\chi|+ |\nabla \chi|\ls C(N,K,R)$.
\end{lem}

 \subsection{Sobolev spaces} $\  $

Several different notions of Sobolev spaces on metric measure space $(X,d,\mu)$  have been established in
 \cite{che99,shan00,hk00,haj03}. They are equivalent to each other on  $RCD^*(K,N)$ metric measure spaces (see, for example, \cite{ags13-lip}).

 Let $(X,d,\mu)$ be a metric measure space satisfying $RCD^*(K,N)$ for some $K\in\mathbb R$ and $N\in[1,\infty)$.
Fix  an open set $\Omega$ in $X$.  We denote by $Lip_{\rm loc}(\Omega)$ the set of locally Lipschitz continuous functions on $\Omega$, and by $Lip(\Omega)$ (resp. $Lip_0(\Omega)$) the set of Lipschitz continuous functions on $\Omega$ (resp, with compact support in $\Omega$).

 Let $\Omega\subset X$ be an open set. For any $1\ls p\ls +\infty$ and
   $f\in Lip_{\rm loc}(\Omega)$, its $W^{1,p}(\Omega)$-norm is defined by
$$\|f\|_{W^{1,p}(\Omega)}:=\|f\|_{L^{p}(\Omega)}+\|{\rm Lip}f\|_{L^{p}(\Omega)}.$$
 The Sobolev spaces $W^{1,p}(\Omega)$ is defined by the closure of the set
$$\big\{f\in Lip_{\rm loc}(\Omega)|\ \|f\|_{W^{1,p}(\Omega)}<+\infty\big\}$$
under  the $W^{1,p}(\Omega)$-norm. Remark that $W^{1,p}(\Omega)$ is reflexive for any $1<p<\infty $ (see \cite[Theorem 4.48]{che99}).
Spaces $W_0^{1,p}(\Omega)$ is defined by the closure of $Lip_0(\Omega)$ under  the $W^{1,p}(\Omega)$-norm.
We say a function $f\in W^{1,p}_{\rm loc}(\Omega)$ if $f\in W^{1,p}(\Omega')$ for every open subset $\Omega'\subset\subset\Omega.$

The following  two facts are  well-known for experts. For the convenience of readers, we include a proof here.
\begin{lem}
(i) For any $1<p<\infty $, we have  $W^{1,p}(X)=W^{1,p}_0(X)$.\\
 (ii) $W^{1,2}(X)=D({\rm Ch})$.
\end{lem}
\begin{proof} If  $X$ is compact, the assertion (i) is clear. Without loss of generality, we can assume that $X$ is non-compact.
Given a function $f\in Lip(X)\cap W^{1,p}(X)$, in order to prove (i), it suffices to find a sequence $(f_j)_{j\in\mathbb N}$ of Lipschitz functions with compact supports in $X$ such that $f_j\to f$ in $W^{1,p}(X).$

  Consider a family of Lipschitz cut-off $\chi_j$ with, for each $j\in\mathbb N$, $\chi_j(x)=1$ for $x\in B_j(x_0)$ and $\chi_j(x)=0$ for $x\not\in B_{j+1}(x_0)$, and $ 0\ls\chi_j(x)\ls1, |\nabla\chi_j| (x)\ls1$ for all $x\in X$. Now  $f\cdot\chi_j\in Lip_0(X)$ and $f\cdot\chi_j(x)\to f(x)$ for $\mu$-almost all $x\in X$.
Notice that $|f\cdot\chi_j|\ls |f|\in L^p(X)$ for all $j$, the dominated convergence theorem implies that  $f\cdot\chi_j\to f $ in  $L^p(X)$ as $j\to\infty$. On the other hand,  since
$$|\nabla (f\cdot\chi_j)|\ls |\nabla f|\cdot \chi_j+|f|\cdot |\nabla \chi_j| \ls |\nabla f| +|f| \in L^p(X)$$
for all $j\in \mathbb N$, we obtain that the sequence $(f\cdot\chi_j)_{j\in\mathbb N}$ is bounded in $W^{1,p}(X).$ By noticing that  $W^{1,p}(X)$ is reflexive (see \cite[Theorem 4.48]{che99}), we can see  that $f\cdot\chi_j$ converges weakly  to $f$ in  $W^{1,p}(X)$ as $j\to\infty.$ Hence,
by Mazur¡¯s lemma, we conclude that there exists a convex combination of  $f\cdot\chi_j$ converges strongly to  $f$ in  $W^{1,p}(X)$ as $j\to\infty.$ The proof of (i) is completed.

Let us prove (ii). It is obvious that $W^{1,2}(X)\subset D({\rm Ch})$, since $Lip(X)\cap W^{1,2}(X)\subset D({\rm Ch})$ and $|\nabla f_n|\leq Lip(f_n)$. We need only to show $D({\rm Ch})\subset W^{1,2}(X)$. This follows immediately from the fact that Lipschitz functions are
 dense in $ D(\rm Ch)$. The proof of (ii) is completed.
 \end{proof}

\section{The weak Laplacian and a local version of Bochner formula}

 Let $(X,d,\mu)$ be a metric measure space satisfying $RCD^*(K,N)$ for some $K\in\mathbb R$ and $N\in[1,\infty)$.
Fix any open set $\Omega\subset X$.
We will denote by the Sobolev spaces $H^1_0(\Omega):=W^{1,2}_0(\Omega)$, $H^1(\Omega):=W^{1,2}(\Omega) $ and $H^1_{\rm loc}(\Omega):=W_{\rm loc}^{1,2}(\Omega).$

\begin{defn}[Weak Laplacian]\label{laplace}
Let $\Omega\subset X$ be an open set, the \emph{Laplacian} on $\Omega$ is an operator $\mathscr L$ on $ H^1(\Omega)$ defined as the follows. For each  function $f\in H^1(\Omega)$, its Lapacian $\mathscr Lf$ is a functional acting on  $H^1_0(\Omega)\cap L^\infty(\Omega)$ given by
$$\mathscr Lf(\phi):=-\int_\Omega\ip{\nabla f}{\nabla \phi}d\mu\qquad \forall\ \phi\in H^1_0(\Omega)\cap L^\infty(\Omega).$$
For any $g\in H^{1}(\Omega) \cap L^{\infty}(\Omega)$, the distribution $g\cdot\mathscr Lf$ is  a functional acting on  $H^1_0(\Omega)\cap L^\infty(\Omega)$ defined by
\begin{equation}\label{eq3.1}
g\cdot\mathscr Lf(\phi):= \mathscr Lf(g\phi)\qquad \forall\ \phi\in H^1_0(\Omega)\cap L^\infty(\Omega).
\end{equation}
\end{defn}
 This Laplacian (on $\Omega$) is linear due to that  the inner product  $\ip{\nabla f}{\nabla g}$ is linear.
The strongly local property of the inner product $\int_X\ip{\nabla f}{\nabla g}d\mu$ implies that if  $f\in H^1(X)$ and $f=  constant$ on $\Omega$  then $\mathscr Lf(\phi)=0 $ for any $\phi\in H^1_0(\Omega)\cap L^\infty(\Omega).$

If, given  $f\in H^{1}(\Omega)$, there exists a function $u_f\in L^1_{\rm loc}(\Omega)$ such that
\begin{equation}\label{eq3.2}
\mathscr Lf(\phi)=\int_\Omega u_f\cdot\phi d\mu\qquad \forall\ \phi\in H^1_0(\Omega)\cap L^\infty(\Omega),
\end{equation}
then we say that  ``$\mathscr Lf$  is a function in  $ L^1_{\rm loc}(\Omega)$" and   write as ``$\La f=u_f$ in the sense of distributions". It is similar to say that  ``$\mathscr Lf$  is  a function in  $ L^p_{\rm loc}(\Omega)$ or $W^{1,p}_{\rm loc}(\Omega)$ for any $p\in[1,\infty]$", and so on.

The operator $\mathscr L$ satisfies the following Chain rule and Leibniz rule, which is essentially due to Gigli \cite{gig15}.
\begin{lem}\label{lem3.2}
 Let $\Omega$ be an  open domain of a metric measure space $(X,d,\mu)$ satisfying $RCD^*(K,N)$ for some $K\in \mathbb R$ and $N\in[1,\infty)$.\\
{\rm(i)\ (Chain rule)}\ \ \  Let $f\in  H^{1}(\Omega)\cap L^\infty(\Omega)$ and $\eta\in C^2(\mathbb R)$. Then we have
\begin{equation}\label{chain}
\La[\eta(f)]=\eta'(f)\cdot\La f+\eta''(f)\cdot|\nabla f|^2.
\end{equation}
{\rm(ii)\ (Leibniz rule)}\ \ \ Let $f,g\in  H^{1}(\Omega)\cap L^\infty(\Omega)$. Then we have
\begin{equation}\label{leibniz}
\La(f\cdot g)=f\cdot\La g+g\cdot\La f+2\ip{\nabla f}{\nabla g}.
\end{equation}
\end{lem}
\begin{proof}The proof is given essentially in \cite{gig15}. For the completeness, we sketch it. We prove only the Chain rule (\ref{chain}). The proof  of   Leibniz rule (\ref{leibniz}) is similar.

Given any $\phi\in H^1_0(\Omega)\cap L^\infty(\Omega)$, we have
$$\La[\eta(f)](\phi)=-\int_\Omega\ip{\nabla[\eta(f)]}{\nabla \phi}d\mu=-\int_\Omega\eta'(f)\cdot\ip{\nabla f}{\nabla \phi}d\mu,$$
where we have used that  (see \cite[\S 3.3]{gig15}) the inner product $\ip{\nabla f}{\nabla \phi}$  satisfies the Chain rule, i.e., $\ip{\nabla[\eta(f)]}{\nabla \phi}=\eta'(f)\cdot\ip{\nabla f}{\nabla \phi}$.

On the other hand, by (\ref{eq3.1}), we obtain
\begin{equation*}
\begin{split}
\Big[\eta'(f)\cdot\La f+\eta''(f)\cdot|\nabla f|^2\Big](\phi)&=\La f\big(\eta'(f)\cdot \phi\big)+\int_\Omega\eta''(f)\cdot|\nabla f|^2\cdot  \phi d\mu\\
&=-\int_\Omega\ip{\nabla f}{ \nabla \big(\eta'(f)\cdot \phi\big)}d\mu  +\int_\Omega\eta''(f)\cdot|\nabla f|^2\cdot  \phi d\mu\\
&=-\int_\Omega\ip{\nabla f}{ \nabla \phi}\cdot\eta'(f)d\mu,
\end{split}
\end{equation*}
where we have used that  $\eta'(f)\cdot\phi\in H^1_0(\Omega)\cap L^\infty(\Omega)$ and that (see \cite[\S 3.3]{gig15}) the inner product $\ip{\nabla f}{\nabla g}$  satisfies  the Chain rule and Leibniz rule, i.e.,
\begin{equation*}\begin{split}
\ip{\nabla f}{ \nabla \big(\eta'(f)\cdot \phi\big)}&=\ip{\nabla f}{ \nabla  \phi }\eta'(f)+\ip{\nabla f}{ \nabla \big(\eta'(f)\big)} \phi\\
&=\ip{\nabla f}{ \nabla  \phi }\eta'(f)+\ip{\nabla f}{ \nabla f}\cdot\eta''(f) \phi.
\end{split}
\end{equation*}
The combination of the above two equations implies the Chain rule  (\ref{chain}). The proof is completed.
\end{proof}

To compare   the above  Laplace operator  $\La $ on $X$ with the generator $\Delta_{\E}$ of the canonical Dirichlet form $(\mathscr E,\mathbb V)$, it was shown \cite{gig15} that the following compatibility result holds.
\begin{lem}[Proposition 4.24 in \cite{gig15}]\label{lem3.3}
The following two statements are equivalent:\\
\indent $i)\quad $  $f\in H^1(X)$ and $\La f$ is a function in $L^2(X)$,\\
\indent $ii)\quad $  $f\in D(\Delta_{\E})$.\\
In each of these cases, we have $\La f=\Delta_{\E}f.$
\end{lem}

The following regularity result for the Poisson equation has been proved under a Bakry-Emery type heat semigroup curvature condition, which is implied by the Riemannian curvature-dimension condition $RCD^*(K,N)$ (see \cite[Theorem 7]{eks15} and \cite[Theorem 7.5]{ams15}).
\begin{lem}[\cite{jia14,jky14}]\label{lip}
Let $(X,d,\mu)$  be a  metric measure space  satisfying   $RCD^*(K,N)$ for $K\in\mathbb R$ and $N\in[1,\infty)$. Let $g\in L^\infty(B_R)$, where $B_R$ is a geodesic ball with radius $R$ and centered at a fixed point $x_0$.  Assume $f\in H^1(B_R)$ and $\La f=g$ on $B_R$ in the sense of distributions. Then we have
$|\nabla f|\in L^\infty_{\rm loc}(B_R),$ and
$$\||\nabla f|\|_{ L^\infty(B_{R/2})}\ls C(N,K,R)\cdot \Big(\frac{1}{\mu(B_R)}\|f\|_{ L^1(B_{R})}+\|g\|_{ L^\infty(B_{R})}\Big).$$
\end{lem}
\begin{proof} In the case of $g=0$, i.e., $f$ is harmonic on $B_R$, the assertion is proved in \cite[Theorem 1.2]{jia14} (see also \cite[Theorem 3.9]{gmo15}). In the general case $g\in L^\infty(\Omega)$, this is proved in \cite[Theorem 3.1]{jky14}. The assertion of the constant $C(N,K,R)$ depending only on $N,K,R$ comes from the fact that both the doubling constant and $L^2$-Poincare constant on a ball $B_R$ of a $RCD^*(K,N)$-space depend on $N,K$ and $R$.
\end{proof}

Now we will give a local version of the Bochner formula, Theorem \ref{thm-bochner}, by combining the modified Bochner formula (\ref{modi-bochner}) and a similar argument in \cite{hkx13,jia15}.
 \begin{thm}[\cite{ams15,hkx13}]\label{theorem-boch}
  Let $(X,d,\mu)$  be a  metric measure space  satisfying   $RCD^*(K,N)$ for $K\in\mathbb R$ and $N\in[1,\infty)$. Let $B_R$ be a geodesic ball with radius $R$ and centered at a  point $x_0$.

  Assume that $f\in H^{1}(B_R)$ satisfies
$\mathscr Lf=g $ on $B_R$ in the sense of distributions with the function $g\in H^1(B_R)\cap L^\infty(B_R)$.
Then we have $|\nabla f|^2\in H^{1}(B_{R/2}) \cap L^\infty(B_{R/2})$  and
\begin{equation}\label{eq3.5}
 \frac{1}{2}\La(|\nabla  f |^2 ) \gs  \Big[\frac{g^2}{N}+\ip{\nabla f}{\nabla  g}+K|\nabla f|^2\Big]\cdot\mu \quad {\rm on}\ \   B_{R/2}
 \end{equation}
in the sense of distributions, i.e.,
\begin{equation*}
-\frac{1}{2}\int_{B_{R/2}}\ip{\nabla|\nabla  f |^2}{\nabla \phi}d\mu \gs  \int_{B_{R/2}}\phi \cdot\Big( \frac{g^2}{N}+\ip{\nabla f}{\nabla  g}+K|\nabla f|^2\Big) d\mu
\end{equation*}
for any $0\ls \phi\in  H^1_0(B_{R/2})\cap L^\infty(B_{R/2})$.
\end{thm}
\begin{proof}From Lemma \ref{lip} and that $g\in L^\infty(B_R)$, we know  $|\nabla f|\in L^\infty_{\rm loc}(B_R)$.

  We take a cut-off $\chi$  satisfying (i) and (ii) in Lemma \ref{lem2.4}. Let
\begin{equation*}
\widetilde{f}(x):=
\begin{cases} f\cdot \chi &{\rm if}\quad x\in B_{R} \\
0 &{\rm if}\quad x\in X \backslash B_{R}.
\end{cases}
\end{equation*}
 Then we have  $\widetilde{f}\in Lip_0(B_R)$.
It is easy to check ${\rm supp}(\La \widetilde{f})\subset B_R$. In fact,  given any $\psi\in H^1_0(X)$ with $\psi=0$ on $B_R$, the strongly local property implies that
$\int_X\ip{\nabla \widetilde{f}}{\nabla \psi}d\mu=0.$

Now we want to calculate $\La \widetilde{f}$ on $B_R$. By the Leibniz rule (\ref{leibniz}), we have, on $B_R$,
\begin{equation*}
\begin{split}
 \La\widetilde{f}=\La (f\cdot \chi)&=\chi\cdot\La f+f\cdot\La \chi+2\ip{\nabla f}{\nabla \chi}\\
 & =\chi\cdot g+f\cdot\Delta_{\mathscr E}\chi+2\ip{\nabla f}{\nabla \chi}\in L^\infty_{\rm loc}(B_R),
 \end{split}
 \end{equation*}
where we have used $g\in L^\infty(B_R)$ and $|\nabla f|\in L^\infty_{\rm loc}(B_R)$, and  that $\chi,|\nabla\chi|,\Delta_{\mathscr E}\chi\in L^\infty(X)$ in Lemma \ref{lem2.4}. Combining with ${\rm supp}(\La \widetilde{f})\subset B_R$, we have  $\La\widetilde{f} \in L^2(X)\cap L^\infty(X)$.
Therefore, by Lemma \ref{lem3.3},
 we get $\widetilde{f}\in D(\Delta_{\mathscr E}) $ and
 \begin{equation}\label{eq3.6}
L^2(X)\cap L^\infty(X)\ni\Delta_{\mathscr E}\widetilde{f}=\La \widetilde{f}=
\begin{cases} \chi\cdot g+f\cdot\Delta_{\mathscr E}\chi+2\ip{\nabla f}{\nabla \chi} &{\rm if}\quad x\in B_{R} \\
0 &{\rm if}\quad x\in X \backslash B_{R}.
\end{cases}
\end{equation}
According to  Lemma \ref{lem-boch}(ii) and $0\ls \phi\in  H^1_0(B_{R/2})\subset\mathbb V$, we conclude that
$|\nabla \widetilde{f}|^2\in \mathbb{V}$ and that
\begin{equation*}
\begin{split}
 \int_X\Big(-\frac{1}{2} \ip{\nabla |\nabla \widetilde{f}|^2}{\nabla \phi}+&\Delta_{\mathscr E}\widetilde{f}\cdot \ip{\nabla \widetilde{f}}{\nabla \phi}+ \phi\cdot(\Delta_{\mathscr E} \widetilde{f})^2\Big)d\mu\\
  &\gs  \frac{1}{N} \int_X\phi\cdot(\Delta_{\mathscr E} \widetilde{f})^2d\mu+K\int_X\phi|\nabla \widetilde{f}|^2d\mu.\qquad
\end{split}
\end{equation*}
Since  $\widetilde{f}=f$ on  $B_{R/2}$, we have $|\nabla f |=|\nabla \widetilde{f}|$ for $\mu$-a.e. on $B_{R/2}$. Notice that $|\nabla \widetilde{f}|^2\in \mathbb{V}$ implies that $|\nabla \widetilde{f}|^2\in H^1(B_{R/2}).$ Then $|\nabla f|^2\in H^1(B_{R/2}),$  and $|\nabla|\nabla f|^2|=|\nabla|\nabla \widetilde{f}|^2|$ in $L^2(B_{R/2}).$
 By (\ref{eq3.6}) and that $|\nabla \chi|=\Delta_{\mathscr E}\chi=0$ on $B_{R/2}$ (since $\chi=1$ on $B_{2R/3}$), we have $ \Delta_{\mathscr E} \widetilde{f}=g$ on $B_{R/2}$.
Hence, we obtain
\begin{equation}\label{eq3.7}
\begin{split}
 \int_{B_{R/2}}\Big(-\frac{1}{2} &\ip{\nabla |\nabla f|^2}{\nabla \phi}+g\cdot \ip{\nabla f}{\nabla \phi}+\phi\cdot g^2\Big)d\mu\\
  &\gs \frac{1}{N}\int_{B_{R/2}}\phi\cdot g^2d\mu+K\int_{B_{R/2}}\phi|\nabla f|^2d\mu.\qquad
\end{split}
\end{equation}
Noticing that $g\cdot \phi\in  H^1_0(B_{R/2})\cap L^\infty(B_{R/2})$ and $\La f=g$ on $B_R$ in the sense of distributions, we have
\begin{equation*}
\begin{split}
\int_{B_{R/2}}g\cdot g  \phi d\mu&=\La f(g  \phi)=-\int_{B_{R/2}}\ip{\nabla f}{\nabla ( g  \phi )} d\mu \\
&=-\int_{B_{R/2}}\ip{\nabla f}{\nabla   g}\cdot \phi  d\mu-\int_{B_{R/2}}\ip{\nabla f}{\nabla\phi}\cdot g d\mu.
\end{split}
\end{equation*}
By combining this and (\ref{eq3.7}), we get the desired inequality (\ref{eq3.5}). The proof is finished.
\end{proof}

By using the same argument of \cite{bq00}, one can get an improvement of the above Bochner formula. One can also consult a detailed argument given in \cite[Lemma 2.3]{jz16}.
\begin{cor}\label{lem6.1}

 Let $(X,d,\mu)$  be a  metric measure space  satisfying   $RCD^*(K,N)$ for $K\in\mathbb R$ and $N\in[1,\infty)$. Let $B_R$ be a geodesic ball with radius $R$ and centered at a fixed point $x_0$.

  Assume that $f\in H^{1}(B_R)$ satisfies
$\mathscr Lf=g $ on $B_R$ in the sense of distributions with the function $g\in H^1(B_R)\cap L^\infty(B_R)$.
Then we have $|\nabla f|^2\in H^{1}(B_{R/2}) \cap L^\infty(B_{R/2})$  and that the distribution $\La(|\nabla f|^2)$ is a signed Radon measure on  $B_{R/2}$. If its Radon-Nikodym decomposition w.r.t. $\mu$ is denoted by
$$\La(|\nabla f|^2)=  \La^{\rm ac}(|\nabla f|^2)\cdot\mu+\La^{\rm sing}(|\nabla f|^2),$$
then we have  $\La^{\rm sing}(|\nabla f|^2)\gs0$ and, for $\mu$-a.e. $x\in B_{R/2},$
$$ \frac{1}{2}\La^{\rm ac}(|\nabla  f |^2 ) \gs  \frac{g^2}{N}+\ip{\nabla f}{\nabla  g}+K|\nabla f|^2.$$
Furthermore, if $N>1$,
 for  $\mu$-a.e. $  x\in B_{R/2}\cap \big\{y:\ |\nabla f(y)|\not=0\big\}$,
\begin{equation}\label{eq3.8}
\frac{1}{2} \La^{\rm ac}(|\nabla f|^2)\gs \frac{ g^2}{N}+ \ip{\nabla f}{\nabla g}+K|\nabla f|^2+\frac{N}{N-1}\cdot\Big(\frac{\ip{ \nabla f}{\nabla |\nabla f|^2}}{2|\nabla f|^2}-\frac{ g}{N}\Big)^2.
\end{equation}
\end{cor}

\section{The maximum principle}
Let $K\in\mathbb R$ and $N\in [1,\infty)$ and let $(X,d,\mu)$  be  a metric measure space satisfying $RCD^*(K,N)$.  In this section, we will study the maximum principle on $(X,d,\mu)$. Let us begin from the Kato's inequality for \emph{weighted measures}.
\subsection{The Kato's inequality} $\ $

Let  $\Omega$  be a bounded  open set of $(X,d,\mu)$. Fix  any $w\in  H^{1}(\Omega)\cap L^\infty(\Omega)$, we consider the weighted measure
$$\mu_w:=e^w\cdot\mu\quad {\rm on}\ \ \Omega.$$
Since, the density $e^{-\|w\|_{L^\infty(\Omega)}}\ls e^w\ls e^{\|w\|_{L^\infty(\Omega)}}$ on $\Omega$, we know that the associated the Lebesgue space $L^p(\Omega,\mu_w)$ and the Sobolev spaces $W^{1,p}(\Omega,\mu_w)$ are equivalent to the original $L^p(\Omega) $ and $W^{1,p}(\Omega)$, respectively, for all $p\gs1.$
Both the measure doubling property and the $L^2$-Poincare inequality still hold with respect to this measure $\mu_w$ (the constants, of course, depend on $\|w\|_{L^\infty(\Omega)}$).

For this measure $\mu_w$, we defined the associated Laplacian $\La_w$ on $f\in H^1(\Omega)$ by
$$\La_wf(\phi):=-\int_\Omega\ip{\nabla f}{\nabla \phi}d\mu_w\  \ \Big(=-\int_\Omega\ip{\nabla f}{\nabla \phi}e^wd\mu\Big)$$
for any $\phi\in H^1_0(\Omega)\cap L^\infty(\Omega).$ It is easy to check that
$$\La_wf=e^w\cdot\La f+e^w\cdot\ip{\nabla w}{\nabla f}$$
in the sense of distributions, i.e., $\La_wf(\phi)= \La f(e^w\cdot\phi)+ \int_\Omega\ip{\nabla w}{\nabla f}  e^w\cdot\phi d\mu.$

When $\Omega$ be a domain of the Euclidean space $\mathbb R^N$ with dimension $N\gs1$, the classical Kato's inequality  states that given any function $f\in L^1_{\rm loc}(\Omega)$  such that $\Delta f\in L^1_{\rm loc}(\Omega)$, then
$\Delta f_+$  is a signed Radon measure and the following holds:
$$\Delta f_+\gs \chi[f\gs0]\cdot\Delta f$$
in the sense of distributions, where $f_+:=\max\{f,0\}.$ Here,
$\chi[f\gs0](x)=1$ for $x$ such that $f(x)\gs0$ and $\chi[f\gs0](x)=0$ for $x$ such that $f(x)<0.$
In \cite{bp04}, the result was extended to  the case when $\Delta f$ is a signed Radon measure.

In the following, we will extend the Kato's inequality to the  metric measure spaces $(X,d,\mu_w)$, under assumption $f\in H^1(\Omega)$.
\begin{prop}[Kato's inequality]\label{kato}
Let  $ \Omega $  be a bounded  open set of $(X,d,\mu)$ and let $w\in  H^{1}(\Omega)\cap L^\infty(\Omega)$. Assume that $f\in H^1(\Omega)$ such that $\La_wf $ is a signed Radon measure.
Then
$\La_w f_+$  is a signed Radon measure and the following holds:
\begin{equation}\label{kato-ineq}
\La_w f_+\gs \chi[f\gs0]\cdot\La_w  f  \quad {\rm on}\ \ \Omega,
\end{equation}
in the sense of distributions. In the sequel, we denote the Radon-Nikodym decomposition $\La_wf=\La_w^{\rm ac}f\cdot\mu_w+\La_w^{\rm sing}f.$
\end{prop}
\begin{proof}
It suffices to prove the following equivalent property:
\begin{equation}\label{eq4.2}
\La_w|f|\gs {\rm sgn}(f)\cdot\La_w f,
\end{equation}
where ${\rm sgn}(t)=1$ for $t>0$, ${\rm sgn}(t)=-1$ for $t<0$, and ${\rm sgn}(t)=0$ for $t=0.$

Fix any $\epsilon>0$ and let
$$f_\epsilon(x):=(f^2+\epsilon^2)^{1/2}\gs \epsilon.$$
We have $f^2_\epsilon=f^2+\epsilon^2$,
\begin{equation}\label{eq4.3}
 |\nabla f_\epsilon|=\frac{|f|}{f_\epsilon}|\nabla f|\ls |\nabla f|
 \end{equation}
and
$$2f_\epsilon\cdot \La_wf_\epsilon+2|\nabla f_\epsilon|^2=\La_w f^2_\epsilon=\La_wf^2=2f\cdot \La_wf+2|\nabla f|^2.$$
Thus,
\begin{equation}\label{eq4.4}
\La_wf_\epsilon\gs \frac{f}{f_\epsilon}\cdot\La_wf,
\end{equation}
 Notice that $|\nabla f_\epsilon|\ls |\nabla f|$  and $f_\epsilon \to |f|$ in $L^2(\Omega)$ implies that $f_\epsilon$ is bounded in $H^1(\Omega)$ and, hence, there exists a subsequence $f_{\epsilon_j}$ converging weakly to $|f|$ in $H^1(\Omega)$. Thus, the measures $\La_w(f_{\epsilon_j})$ converges weakly to $\La_w|f|.$
On the other hand, notice that  $f_\epsilon(x)\to |f(x)|$ for each $x\in\Omega$ and that $|f/f_\epsilon|\ls1$  on $\Omega$. Letting $\epsilon:=\epsilon_j\to 0$ in (\ref{eq4.4}),  we conclude that
$$\La_w(|f|) \gs \frac{f}{|f|}\cdot \La_w f.$$
This is (\ref{eq4.2}), and the proof is completed.
\end{proof}

\subsection{Maximum principles}$ \ $

 The above Kato's inequality implies the maximum principle Theorem \ref{thm1.3}. Precisely, we have the following.
 \begin{thm}\label{max-elliptic}
Let $\Omega$ be a bounded domain.
Let $f(x)\in H^{1}(\Omega)\cap L^\infty_{\rm loc}(\Omega)$  such that  $\La f $ is a signed Radon measure with $\La^{\rm sing}f\gs 0$.
Suppose that $f$ achieves one of its   strict maximum in  $\Omega$ in the sense that: there exists a neighborhood $U\subset\subset \Omega$ such that
\begin{equation}\label{equation4.5}
\sup_{U}f>\sup_{\Omega\backslash U}f.
\end{equation}
Here and in the sequel of the paper, the notion $\sup_{U}f$ means always ${\rm ess}\sup_{U}f.$ Then, given any $w\in  H^{1}(\Omega)\cap L^\infty(\Omega)$, for any $\varepsilon>0$, we have
\begin{equation}\label{equation4.6}
\mu\Big\{x:\ f(x)\gs\sup_\Omega f-\varepsilon\ \ {\rm and}\ \ \La^{\rm ac}f(x)+\ip{\nabla f}{\nabla w}(x)\ls \varepsilon \Big\} >0.
\end{equation}
In particular, there exists a sequence of points $\{x_j\}_{j\in\mathbb N}\subset U$  such that they are the approximate continuity points of $\La^{\rm ac}f$ and $\ip{\nabla f}{\nabla w}$, and that
$$f(x_j)\gs \sup_\Omega f-1/j\qquad{\rm and }\qquad  \La^{\rm ac}f(x_j)+\ip{\nabla f}{\nabla w}(x_j)\ls 1/j.$$
\end{thm}
 \begin{proof}
Suppose the first assertion (\ref{equation4.6}) fails for some sufficiently small $\varepsilon_0>0.$ Then we have  $\big(f-(\sup_\Omega f-\varepsilon_0)\big)_+\in H^1_0(\Omega)$ (by  the maximal property (\ref{equation4.5})) and
$$\mu\Big\{x:\ f(x)\gs\sup_\Omega f-\varepsilon_0\ \ {\rm and}\ \ \La^{\rm ac}f+\ip{\nabla f}{\nabla w}\ls \varepsilon_0 \Big\} =0.$$
Then for almost $x\in\{y:\ f(y)\gs \sup_\Omega f-\varepsilon_0\}$ we have
$$\La_w^{\rm ac}f(x)\cdot\mu_w=e^{w(x)}\cdot (\La^{\rm ac}f+\ip{\nabla f}{\nabla w})(x)\cdot\mu>e^{-\|w\|_{L^\infty}}\varepsilon_0\cdot\mu>0.$$
The assumption $\La^{\rm sing}f\gs 0$ implies that $\La_w^{\rm sing}f\gs 0$.
By applying the Proposition \ref{kato} to the function $f-(\sup_\Omega-\varepsilon_0)$, we have
$$\La_w \big(f-(\sup_\Omega f-\varepsilon_0)\big)_+\gs \chi[f\gs \sup_\Omega f-\varepsilon_0]\cdot\La_w^{\rm ac}f\cdot\mu_w\gs0$$
 on $\Omega,$ in the sense of distributions. Recall that the metric measure space $(X,d,\mu_w)$  satisfies a doubling property and supports a $L^2$-Poincare inequality.  Now the weak maximum principle \cite[Theorem 7.17]{che99} implies that $\big(f-(\sup_\Omega f-\varepsilon_0)\big)_+=0$ on $\Omega.$ Thus, $\sup_\Omega f\ls  \sup_\Omega f-\varepsilon_0$ on $\Omega$. This is a contradiction, and proves the first assertion (\ref{equation4.6}).

The second assertion follows from the first one by taking $\varepsilon=1/j.$
\end{proof}

Next, let us consider the parabolic version of the maximum principle.
We need  the following  parabolic  weak maximum principle.
\begin{lem}\label{lemma4.3}
Let $\Omega$ be a bounded open subset and let $T>0$. Let $w\in  H^{1}(\Omega_T)\cap L^\infty(\Omega_T)$ with $\partial_tw(x,t)\ls C$ for some constant $C>0$, for almost all $(x,t)\in\Omega_T$.
Suppose that  $f(x,t)  \in H^{1}(\Omega_T)\cap L^\infty(\Omega_T)$ with
$ \lim_{t\to0}\|f(\cdot,t)\|_{L^2(\Omega)}=0$ and, for almost all $t\in(0,T)$, that the functions $f(\cdot,t)\in H^1_0(\Omega)$.
  Assume that, for almost every  $t\in (0,T)$, the function $f(\cdot,t)$ satisfies
\begin{equation}\label{eq4.7}
\La_{w(\cdot,t)} f(\cdot,t) -\frac{\partial}{\partial t} f(\cdot,t)\cdot\mu_{w(\cdot,t)}  \gs   0\ \quad   {\rm on}\ \ \Omega
\end{equation}
Then we have
$$\sup_{\Omega\times(0,T)}f(x,t) \ls0.$$
\end{lem}
\begin{proof}
The proof is standard via a Gaffney-Davies' method  (see also  \cite[Lemma 1.7]{stu95}).   We include a proof here for the completeness.  Since $f_+$ meets all of conditions in this lemma, by replacing $f$ by $f_+$, we can assume that $f\gs0.$

Put
$$\xi(t):=\int_\Omega f^2(\cdot,t)d\mu_{w(\cdot,t)}.$$
Since $\mu_{w(\cdot,t)}=e^{w}\cdot\mu\ls e^{\|w\|_{L^\infty}}\cdot\mu$ and $f\in H^{1}(\Omega_T)$, we have, for almost all $t\in(0,T)$,
\begin{equation*}
\begin{split}
\xi'(t)&=\int_{\Omega}\partial_t(f^2)d\mu_{w(\cdot,t)}+\int_{\Omega} f^2\cdot\partial_tw\cdot d\mu_{w(\cdot,t)}\\
&\ls -2\int_{\Omega}|\nabla f|^2d\mu_{w(\cdot,t)}+C\cdot\xi(t)\ls C\cdot\xi(t),
\end{split}
\end{equation*}
where we have used $\partial_tw\ls C$ and that the functions $f(\cdot,t)\in H^1_0(\Omega)\cap L^\infty(\Omega)$ for almost all $t\in(0,T)$.
By using $\lim_{t\to0}\xi(t)=0$ (since $\xi(t)\ls e^{\|w\|_{L^\infty}}\cdot\|f(\cdot,t)\|_{L^2(\Omega)}$ and the assumption $\lim_{t\to0}\|f(\cdot,t)\|_{L^2(\Omega)}=0$), one can obtain that
$\xi(t)\ls0$. This implies $f=0$ almost all in $\Omega_T$. The proof is finished.
\end{proof}
By using the same argument as in Theorem \ref{max-elliptic},  the combination of the Kato's inequality and  Lemma \ref{lemma4.3} implies  the following parabolic maximum principle.
\begin{thm}\label{max-parabolic}
Let $\Omega$ be a bounded domain  and let $T>0$.
Let $f(x,t)\in H^{1}(\Omega_T)\cap L^\infty(\Omega_T)$ and suppose that $f$ achieves one of its strict maximum in $\Omega\times(0,T]$ in the sense that:  there exists a neighborhood $U\subset\subset \Omega$ and an interval $(\delta,T]\subset   (0,T]$ for some $\delta>0$ such that
$$\sup_{U\times(\delta,T]}f>\sup_{\Omega_T\backslash ( U\times(\delta,T])}f .$$
Here $\sup_{U\times(\delta,T]}f$ means ${\rm ess}\sup_{U\times(\delta,T]}f$. Assume that, for almost every $t\in(0,T)$,  $\La f(\cdot, t) $ is a signed Radon measure with $\La^{\rm sing}f(\cdot,t)\gs 0$.
Let  $w\in  H^{1}(\Omega_T)\cap L^\infty(\Omega_T)$  with $\partial_tw(x,t)\ls C$ for some constant $C>0$, for almost all $(x,t)\in\Omega_T$.  Then, for any $\varepsilon>0$, we have
\begin{equation*}%\label{equation4.8}
(\mu\times \mathcal L^1)\Big\{(x,t):\ f(x,t)\gs\sup_{\Omega_T} f-\varepsilon\ \ {\rm and}\ \ \La^{\rm ac}f(x,t)+\ip{\nabla f}{\nabla w}(x,t)-\frac{\partial}{\partial t}f(x,t)\ls \varepsilon \Big\} >0,
\end{equation*}
where $\mathcal L^1$ is the 1-dimensional Lebesgue's measure on $(\delta,T]$.

In particular, there exists a sequence of points $\{(x_j,t_j)\}_{j\in\mathbb N}\subset U\times(\delta,T]$  such that every $x_j$ is  an approximate continuity point of $\La^{\rm ac}f(\cdot, t_j)$ and $\ip{\nabla f}{\nabla w}(\cdot, t_j)$, and that
$$f(x_j,t_j)\gs \sup_{\Omega_T} f-1/j\quad{\rm and }\quad  \La^{\rm ac}f(x_j,t_j)+\ip{\nabla f}{\nabla w}(x_j,t_j)-\frac{\partial}{\partial t}f(x_j,t_j)\ls 1/j.$$
\end{thm}
\begin{proof}
 We will argue by contradiction, which  is similar to the proof of Theorem \ref{max-elliptic}. Suppose the assertion fails for some small $\varepsilon_0>0$. Then, for almost all $(x,t)\in \{(y,s):\ f(y,s)\gs \sup_{\Omega_T}f-\varepsilon_0\},$ we have
 $$\La^{\rm ac}f(x,t)+\ip{\nabla f}{\nabla w}(x,t)-\frac{\partial}{\partial t}f(x,t)\gs \varepsilon_0.$$
 Thus, at such $(x,t)$,
\begin{equation*}
\begin{split}
\Big[\La^{\rm ac}_wf(&x,t) -\frac{\partial}{\partial t}f(x,t)\Big]\cdot\mu_w\\
&\gs \Big[\La^{\rm ac}f(x,t)+\ip{\nabla f}{\nabla w}(x,t)-\frac{\partial}{\partial t}f(x,t)\Big]\cdot e^w\cdot\mu\gs \varepsilon_0\cdot e^w\cdot\mu\gs0.
\end{split}
\end{equation*}
 The strictly maximal property of $f$ gives that $f_{\varepsilon_0}:=\big(f-(\sup_{\Omega_T}f-\varepsilon_0)\big)_+\in H^1(\Omega_T)$ with $ \lim_{t\to0}\|f_{\varepsilon_0}(\cdot,t)\|_{L^2(\Omega)}=0$ and, for almost all $t\in(0,T)$, that the functions $f_{\varepsilon_0}(\cdot,t)\in H^1_0(\Omega)$.   Notice that $\La^{\rm sing}_{w(\cdot,t)}f(\cdot,t)\gs 0$ by $\La^{\rm sing}f(\cdot,t)\gs 0$. By using the Kato's inequality, we have that, for almost every $t\in(0,T)$,
 \begin{equation*}
\begin{split}
\La_w\big(f-&(\sup_{\Omega_T}f-\varepsilon_0)\big)_+\gs \chi[ f\gs(\sup_{\Omega_T}f-\varepsilon_0)]\cdot  \La_w^{\rm ac}f\\
 &\gs  \chi[ f\gs(\sup_{\Omega_T}f-\varepsilon_0)]\cdot  \frac{\partial f}{\partial t}\cdot\mu_w=  \frac{\partial}{\partial t}\big(f-(\sup_{\Omega_T}f-\varepsilon_0)\big)_+\cdot\mu_w.
\end{split}
\end{equation*}
 Then  Lemma \ref{lemma4.3} implies that $\big(f-(\sup_{\Omega_T}f-\varepsilon_0)\big)_+=0$ for almost all $(x,t)\in\Omega_T$. This is a contradiction.
 \end{proof}

\section{Local Li-Yau's gradient estimates}

Let $K\in\mathbb R$ and $N\in [1,\infty)$ and let $(X,d,\mu)$  be  a metric measure space satisfying $RCD^*(K,N)$.  In this section, we will prove the local Li-Yau's gradient estimates--Theorem \ref{thm1.3}.

Let $\Omega\subset X$ be a domain. Given $T>0$, let us still denote
$$\Omega_{T}:=\Omega\times(0,T]$$
 the space-time domain, with lateral boundary $\Sigma$ and parabolic boundary $\partial_P\Omega_{T}:$
$$\Sigma:=\partial \Omega\times(0,T)\quad {\rm and}\quad \partial_P\Omega_{T}:=\Sigma\cup (\Omega \times\{0\}). $$
We adapt the following precise definition of \emph{locally weak solution} for the heat equation.
\begin{defn}\label{defn5.1}
 Let $T\in(0,\infty]$ and let $\Omega$  be a domain.  A function $u(x,t) $ is called a   \emph{locally weak  solution}  of  the heat equation on $\Omega_{T}$  if $u(x,t)  \in H^1(\Omega_T) \ (=W^{1,2}(\Omega_{T}))$
 and if for any  subinterval $[t_1,t_2]\subset(0,T)$ and any geodesic ball $B_R\subset\subset \Omega$,  it holds
\begin{equation}\label{eq5.1}
 \int^{t_2}_{t_1}\int_{B_R}\Big(\partial_t u\cdot\phi+\ip{\nabla u}{\nabla \phi}\Big)d\mu dt = 0
\end{equation}
 for all  test functions $\phi(x,t)\in  Lip_0(B_{R}\times(t_1,t_2)\big).$
Here and in the sequel, we denote always $\partial_tu:=\frac{\partial u}{\partial t}.$
\end{defn}
\begin{rem}\label{rem5.2}
The test functions $\phi$ in this definition can be chosen such that it has to vanish only on the lateral boundary $\partial B_R\times(0,T) $. That is, $\phi \in Lip(B_{R,T})$ with $\phi(\cdot,t)\in Lip_0(B_R)$ for all $t\in(0,T).$

\end{rem}

The local boundedness and the Harnack inequality for locally weak solutions of the heat equation have been established by  Sturm \cite{stu95,stu96} and Marola-Masson  \cite{mm13}. In particular, any locally weak solutions for the heat equation in Definition \ref{defn5.1} must be locally H\"older continuous.

Let $u(x,t)$ be a locally weak solution of  the heat equation on $\Omega\times(0,T).$   Fubini Theorem implies, for a.e. $t\in[0,T]$, that the function $u(\cdot,t)\in H^1(\Omega) $ and $\partial_tu\in L^2(\Omega)$. Hence,  for a.e. $t\in(0,T)$, the function $u(\cdot,t)$ satisfies, in the distributional sense,
\begin{equation}\label{eq5.2}
\La u=\partial_t u \quad {\rm on}\ \ \Omega.
\end{equation}
Conversely, if a  function $u(x,t)\in H^{1}\big(\Omega_T\big)$  and (\ref{eq5.2}) holds for a.e. $t\in[0,T]$, then it was shown
\cite[Lemma 6.12]{zz14} that $u(x,t)$ is a locally weak solution of  the  heat equation on $\Omega_T.$

In the case that $u(x,t)$ is a (globally) weak solution of  heat equation on $X\times(0,\infty)$ with initial value in $L^2(X)$, the theory of analytic semigroups asserts that the function $t\mapsto \|u\|_{W^{1,2}(X)}$ is analytic.  However, for a locally weak solution of the heat equation on $\Omega_{T}$, we have not sufficient regularity for  the time derivative $\partial_t u$: in general, $\partial_tu$ is only in $L^2$. This is not enough to use Bochner formula in Theorem \ref{theorem-boch}  to (\ref{eq5.2}). For overcoming this difficulty, we recall the so-called \emph{Steklov average}.
 \begin{defn}\label{steklov}
 Given a geodesic ball $B_R$ and  a function  $u(x,t)\in L^1(B_{R,T})$, where $B_{R,T}:=B_R\times(0,T)$,  the \emph{Steklov average} of $u$ is defined, for every $\varepsilon\in(0,T)$ and any $h\in  (0,\varepsilon)$, by
\begin{equation}\label{eq5.3}
u_h(x,t):=
\frac{1}{h}\int_0^{h}u(x,t+\tau)d\tau,\quad t\in(0,T-\varepsilon].
\end{equation}
\end{defn}
From the general theory of $L^p$ spaces, we know that
 if $u \in  L^p(B_{R,T})$,  then the Steklov average $u_h$ converges to $u$ in $L^p(B_{R,T-\varepsilon})$  as $h\to 0$, for  every $\varepsilon\in(0,T).$
\begin{lem}\label{lem5.4}
If  $u \in  H^{1}(B_{R,T})\cap L^\infty(B_{R,T})$,  then we have, for every   $\varepsilon\in(0,T) $, that
 $$u_h\in H^{1}(B_{R,T-\varepsilon})\cap L^\infty(B_{R,T-\varepsilon})\quad \ {\rm and}\ \quad
\partial_t u_h\in  H^{1}(B_{R,T-\varepsilon})\cap L^\infty(B_{R,T-\varepsilon})$$
 for  every
$h\in(0,\varepsilon),$ and that $\|u_h\|_{H^{1}(B_{R,T-\varepsilon})}$ is bounded uniformly with respect to $h\in(0,\varepsilon)$.
\end{lem}
\begin{proof}
Since $u \in H^{1}(B_{R,T})$, according to \cite{hk00},
there exists a function $g(x,t)\in L^2(B_{R,T})$ such that
$$|u(x,t)-u(y,s)|\ls d_P\big((x,t),(y,s)\big)\cdot \Big(g(x,t)+g(y,s)\Big),$$
for almost all $(x,t), (y,s)\in B_{R,T}$ with respect to the product measure $d\mu\times dt$,  where $d_P$ is the product metric on $B_{R,T}$ defined by
$$d_P^2\big((x,t),(y,s)\big):=d^2(x,y)+|t-s|^2.$$ Such a function $g$ is called a Haj{\l}asz-gradient of $u$ on $B_{R,T}$ (see \cite[\S8]{haj03}).
By the definition of the Steklov average $u_h$, we have
\begin{equation*}
\begin{split}
|u_h(x,t)-u_h(y,s)|&\ls \frac{1}{h}\int_0^h\Big(g(x,t+\tau)+g(y,s+\tau)\Big)\cdot d_P\big((x,t+\tau),(y,s+\tau)\big)d\tau\\
& = \frac{1}{h}\int_0^h\Big(g(x,t+\tau)+g(y,s+\tau)\Big)d\tau\cdot d_P\big((x,t),(y,s)\big)\\
& = \Big(g_h(x,t)+g_h(y,s)\Big) \cdot d_P\big((x,t),(y,s)\big)
\end{split}
\end{equation*}
for almost all $(x,t), (y,s)\in B_{R,T}$. The fact $g(x,t)\in L^2(B_{R,T})$ implies that $g_h(x,t)\in L^2(B_{R,T-\varepsilon})$ for each $h\in(0,\varepsilon)$ and
that the functions $g_h$ converges to $g$ in $L^2(B_{R,T-\varepsilon})$ as $h\to0$. Then the previous inequality implies that $g_h$ is a Haj{\l}asz-gradient of $u_h$ on $B_{R, T-\varepsilon}$ for all $h\in(0,\varepsilon)$ (see \cite{haj03}). According to \cite[Theorem 8.6]{haj03}, $2g_h$ is a 2-weak
upper gradient of $u_h$.  Thus we conclude that
$u_h\in W^{1,2}(B_{R,T-\varepsilon})$ and
$$\limsup_{h\to0}\int_{B_{R,T-\varepsilon}}(|\nabla u_h|^2+|\partial_t u_h|^2)d\mu dt\ls \limsup_{h\to0}\int_{B_{R,T-\varepsilon}}(2g_h)^2d\mu dt\ls 4\int_{B_{R,T-\varepsilon}}g^2d\mu dt.$$
Therefore, we get that $\|u_h\|_{H^{1}(B_{R,T-\varepsilon})}$ is bounded uniformly with respect to $h\in(0,\varepsilon)$ (by combining with $u_h\to u$ in $L^2(B_{R,T-\varepsilon})$ as $h\to0$).

Lastly,
the assertion $u_h\in L^\infty(B_{R,T-\varepsilon})$ follows directly from the definition of $u_h$ and $u\in L^\infty(B_{R,T})$.   The assertion of $\partial_tu$ follows from that
$$\partial_tu_h=\frac{u(x,t+h)-u(x,t)}{h}.$$
 The proof is completed.
\end{proof}
%\begin{rem}
%Lemma \ref{lem5.4} implies that $u_h$ converges  weakly to $u$ in $H^{1}(B_{R,T-\varepsilon})$. This is enough for our purpose in this paper. However, one might %expect to improve the weak convergence to a strong one:  $u_h\to u$ in $H^{1}(B_{R,T-\varepsilon})$ as $h\to0.$
%For this, one needs only to show
%$$\limsup_{h\to0}\|u_h\|_{H^{1}(B_{R,T-\varepsilon})}\ls \|u\|_{H^{1}(B_{R,T-\varepsilon})}.$$
%However, the  above argument  shows only
%$$\limsup_{h\to0}\|u_h\|_{H^{1}(B_{R,T-\varepsilon})}\ls \|u\|_{H^{1}(B_{R,T})}.$$
%\end{rem}

For a locally weak solution $u$ for the heat equation, we have the following property of $u_h$.
\begin{lem}\label{lem5.6}
Let $u \in  H^{1}(B_{R,T})\cap L^\infty(B_{R,T})$ be a locally weak solution for the heat equation, and fix any two constants $\varepsilon,h$ such that $\varepsilon\in(0,T) $ and
$h\in(0,\varepsilon)$. Then for almost all $t\in (0,T-\varepsilon)$
$$\La u_h=\partial_t  u_h$$
  on $ B_R$, in the  sense of distributions.
\end{lem}
\begin{proof}
The proof is standard. In fact, one can show the assertion for locally Lipschitz function $u$, and then use an approximating argument to prove the lemma.
\end{proof}

With the aid of the above two lemmas, we will consider firstly the case when a locally weak solution $u\in H^{1}(B_{R,T})\cap L^{\infty}(B_{R,T})$ with
$\partial_t u\in H^{1}(B_{R,T})\cap L^{\infty}(B_{R,T})$.
\begin{lem}\label{lemma5.7}
 Given $K\in \mathbb R$ and $N\in[1,\infty)$, let $(X,d,\mu)$ be a metric measure space satisfying $RCD^*(K,N)$.
Let $u(x,t)\in H^{1}(B_{2R,T})\cap L^\infty(B_{2R,T})$ be a  locally weak  solution of the heat equation on $B_{2R,T}$.
Assume that  $\partial_tu\in  H^{1}(B_{2R,T})\cap L^{\infty}(B_{2R,T})$.
Then we have $|\nabla u|^2\in H^1(B_{R,T})\cap L^\infty(B_{R,T}).$
\end{lem}
\begin{proof}
Notice that, for almost all $t\in(0,T)$,  we have  $u(\cdot,t),\partial_tu(\cdot,t)\in H^{1}(B_{2R})\cap L^\infty(B_{2R})$ and  that $\La u=\partial_tu$ on $B_{2R}$. By Lemma \ref{lip}, we get
$$\||\nabla u(\cdot,t)|\|_{L^\infty(B_{3R/2})}\ls C(N,K,R)\cdot (|u(\cdot,t)|_{L^\infty(B_{2R})}+|\partial_t u(\cdot,t)|_{L^\infty(B_{2R})}).$$
This implies $|\nabla u|^2\in L^\infty(B_{3R/2,T})$ and
$$\||\nabla u(\cdot,\cdot)|\|_{L^\infty(B_{3R/2,T})}\ls C(N,K,R)\cdot (|u|_{L^\infty(B_{2R,T})}+|\partial_t u|_{L^\infty(B_{2R,T})}): =C_*.$$

On the other hand,  for almost all $t\in(0,T)$, by applying the Bochner formula  (\ref{eq3.5}) to  $\La u=\partial_tu$  on $B_{2R}$, we  conclude that  $|\nabla u(\cdot,t)|^2\in H^{1}(B_{3R/2})\cap L^\infty(B_{3R/2})$ and
  \begin{equation*}
  \begin{split}
  \La( |\nabla u|^2)&\gs \Big[ 2\frac{(\partial_t u)^2}{N}+2\ip{\nabla u}{\nabla  \partial_tu}+2K|\nabla u|^2\Big]\cdot\mu\\
  &\gs -2|\nabla u|\cdot|\nabla  \partial_tu|\cdot\mu+2K|\nabla u|^2\cdot\mu\ \gs -2\Big[C_*\cdot|\nabla  \partial_tu|+2|K|C^2_*\Big]\cdot\mu,
  \end{split}
  \end{equation*}
 on $ B_{3R/2}$ in the sense of distributions. By using the  Caccioppoli inequality, we conclude that, for almost all $t\in (0,T)$,
 $$\||\nabla|\nabla u|^2(\cdot,t)|\|_{L^2(B_R)}\ls C_{N,K,R}\cdot\big (2C_*\cdot \||\nabla  \partial_tu|\|_{L^2(B_{3R/2})}+2|K|\cdot C^2_* +\||\nabla u|^2\|_{L^2(B_{3R/2})}\big).$$
 The integration on $(0,T)$ implies that
 $$\||\nabla |\nabla u|^2 \|_{L^2(B_{R,T})}\ls C_{**}\cdot\big (  \||\nabla  \partial_tu|\|_{L^2(B_{3R/2},T)}  +\||\nabla u|^2\|_{L^2(B_{3R/2,T})}+1\big),$$
for the constants $C_{**}$ depending on $N,K,R,T $ and $C_*$.  Thus, $|\nabla |\nabla u|^2| \in L^2(B_{R,T})$.

  Lastly, noting that, for almost all $(x,t)\in B_{R,T}$,
$$|\partial_t|\nabla u|^2|^2 =|\partial_t\ip{\nabla u}{\nabla u}|^2=|2\ip{\nabla\partial_t u}{\nabla u}|^2\ls 4 |\nabla\partial_t u|^2\cdot|\nabla u |^2.$$
Then, by using $|\nabla u|^2\in L^\infty(B_{3R/2,T})$ and $\partial_tu\in H^1(B_{R,T})$, we get $|\partial_t|\nabla u|^2|\in L^2(B_{3R/2,T}).$
By combining with  $|\nabla |\nabla u|^2| \in L^2(B_{R,T})$, we conclude $|\nabla u|^2\in H^1(B_{R,T})$. Now we finish the proof.
\end{proof}
\begin{lem}\label{lem5.8}
 Given $K\gs0$ and $N\in[1,\infty)$, let $(X,d,\mu)$ be a metric measure space satisfying $RCD^*(-K,N)$.
Let $u(x,t)\in H^{1}(B_{2R,T})\cap L^\infty(B_{2R,T})$ be the  locally weak  solution of the heat equation on $B_{2R,T}$.
Assume that $u\gs\delta>0$ and $\partial_tu\in  H^{1}(B_{2R,T})\cap L^{\infty}(B_{2R,T})$.
We put
$$F(x,t)=t\cdot[|\nabla f|^2-\alpha\cdot \partial_t f]_+,$$
where $f=\log u$ and $\alpha>1$. Then, we have
$$\frac{F}{t}\in H^{1} (B_{R,T} )\cap L^{\infty}(B_{R,T} ),$$
and that, for almost every $t\in(0,T)$, the function $F(\cdot,t)$ satisfies
\begin{equation}\label{eq5.5}
\La F-\partial_t F\cdot\mu \gs -2\ip{\nabla f}{\nabla F}\cdot\mu-\frac{F}{t} \cdot\mu + 2t\Big[\frac{1}{N}\big(|\nabla f|^2-\partial_tf)^2-K|\nabla f|^2\Big] \cdot\mu\
\end{equation}
on $ B_{R}$, in the sense of distributions.
\end{lem}
\begin{proof}
From Lemma \ref{lemma5.7}, we have  $|\nabla u|^2\in H^{1}(B_{3R/2,T})\cap L^\infty(B_{3R/2,T})$.
 By combining with that $\partial_t u\in L^\infty (B_{2R,T})\cap H^{1}(B_{2R,T})$ and that $u\gs\delta>0$, we get that
$$ |\nabla f|^2-\alpha\partial_tf = \frac{|\nabla u|^2}{u^2} -\alpha\frac{\partial_tu}{u}  \in H^{1}(B_{3R/2,T})\cap L^\infty(B_{3R/2,T}).$$
 This implies  $ F/t= [|\nabla f|^2-\alpha\partial_tf]_+ \in H^{1}(B_{3R/2,T})\cap L^\infty(B_{3R/2,T})$ and  proves the first assertion.

By $\partial_tu\in  H^{1}(B_{2R,T})$, we see that $\partial_{tt}u\in L^2(B_{2R,T})$ and that, for almost all $t\in (0,T)$,
$$\La(\partial_t u)=\partial_{tt}u$$
in the sense of distributions.
Since  $u,\partial_tu\in H^{1}(B_{2R,T})\cap L^\infty(B_{2R,T})$ and $u\gs\delta>0$,
by using the chain rule in Lemma \ref{lem3.2}(i) to both $u$ and $\partial_tu$, we have, for almost all $t\in (0,T)$, that the functions $  f(\cdot,t),\partial_tf(\cdot,t)\in H^{1}(B_{2R})$ and
\begin{equation}\label{eq5.6}
\La f=\partial_tf-|\nabla f|^2,\qquad \La(\partial_tf)= \partial_{tt} f-2\ip{\nabla f}{\nabla \partial_tf}
\end{equation}
on $B_{2R}$ in the sense of distributions.

Consider $F_1(x,t):=t\cdot\partial_tf$. We have, for almost all $t\in(0,T)$, the function $F_1(\cdot, t)\in H^{1}(B_{2R})$ with
$$\La F_1=t\La \partial_tf=t\cdot\big(\partial_{tt}f-2\ip{\nabla f}{\nabla \partial_tf}\big).$$
Noting that
$$\partial_tF_1=\partial_tf+t\partial_{tt}f\quad{\rm and}\quad \ip{\nabla f}{\nabla F_1}=t\ip{\nabla f}{\nabla \partial_tf},$$
we conclude that
\begin{equation}\label{eq5.7}
\La F_1-\partial_tF_1=-2\ip{\nabla f}{\nabla F_1}-\frac{F_1}{t}
\end{equation}
on $B_{2R}$ in the sense of distributions.

Consider $F_2:=t|\nabla f|^2$. Recall that, for almost all $t\in(0,T)$, the function $  f(\cdot,t)\in H^{1}(B_{2R})$ and
$$\partial_tf-|\nabla f|^2=\frac{\partial_tu}{u}-\frac{|\nabla u|^2}{u^2}\in L^\infty(B_{3R/2})\cap H^{1}(B_{3R/2}).$$
 Recalling that $(X,d,\mu)$ satisfies $RCD^*(-K,N)$, we can apply the Bochner formula (\ref{eq3.5}) to $\La f=\partial_tf-|\nabla f|^2$ to conclude that  $|\nabla f|^2\in H^{1}(B_{R})$ and
\begin{equation*}
\La(|\nabla f|^2)\gs 2\Big[\frac{1}{N}\big(\partial_tf-|\nabla f|^2\big)^2+\ip{\nabla f}{\nabla\big(\partial_tf-|\nabla f|^2\big)}
-K|\nabla f|^2\Big]\cdot\mu
\end{equation*}
on $B_{R}$, in the sense of distributions.
Therefore,   for almost all $t\in(0,T)$, we get the function $F_2(\cdot,t)$ satisfies
\begin{equation}\label{eq5.8}
\La F_2-\partial_t F_2\cdot\mu  \gs 2t\cdot\Big[\frac{1}{N}\big(\partial_tf-|\nabla f|^2\big)^2
-K|\nabla f|^2\Big]\cdot\mu -2\ip{\nabla f}{\nabla F_2}\cdot\mu-\frac{F_2}{t}\cdot\mu
\end{equation}
on $B_{R}$, in the sense of distributions.
By combining (\ref{eq5.7}) and (\ref{eq5.8}), we conclude, for almost all $t\in(0,T)$,
that  we have, for $\widetilde{F}:=F_2-\alpha\cdot F_1$,
$$\La \widetilde{F}-\partial_t \widetilde{F} \cdot\mu\gs -2\ip{\nabla f}{\nabla \widetilde{F}}\cdot\mu-\frac{\widetilde{F}}{t} \cdot\mu + 2t\Big[\frac{1}{N}\big(|\nabla f|^2-\partial_tf)^2-K|\nabla f|^2\Big]\cdot\mu. $$
Now, by using the Kato's inequality to $F=\widetilde{F}_+$,   we have the desired estimate (\ref{eq5.5}).
The proof of this lemma is finished.
\end{proof}

We are ready to prove the following local Li-Yau's estimate under some additional assumptions.
\begin{lem}\label{lem5.9}
Given $K\gs0$ and $N\in[1,\infty)$,  let $(X,d,\mu)$ be a metric measure space satisfying $RCD^*(-K,N)$.
Let $T\in(0,\infty)$ and let $u(x,t)\in H^{1}(B_{2R,T})\cap L^\infty(B_{2R,T})$ be a  locally weak  solution of the heat equation on $B_{2R,T}$.
Assume that $u\gs\delta>0$ and $\partial_tu\in  H^{1}(B_{2R,T})\cap L^{\infty}(B_{2R,T})$.

Then, for any $\alpha>1$ and any $\beta,\gamma\in(0,1)$, the following local gradient estimate holds
 \begin{equation}\label{eq5.9}
 \begin{split}
 \sup_{B_R\times(\gamma\cdot T,T] }\Big(|\nabla f|^2-\alpha\cdot \frac{\partial}{\partial t}f\Big)(x,t) \ls& \max\bigg\{1, \frac{1}{2}+\frac{KT}{2(\alpha-1)}\bigg\} \cdot \frac{N\alpha^2}{2T} \cdot\frac{1}{(1-\beta)\gamma}\\
  &\ \  +
\frac{C_N\cdot  \alpha^4}{R^2 (\alpha-1)}\cdot\frac{1}{(1-\beta)\beta\gamma}+ \big(\frac{\sqrt K}{R}+\frac{1}{R^2}\big)\cdot\frac{C_N\cdot\alpha^2}{(1-\beta)\gamma},
 \end{split}
 \end{equation}
  where $f=\ln u$, and $C_{N}$ is a constant depending only on  $N$.
\end{lem}

\begin{proof}
 From the previous Lemma \ref{lem5.8},  we have $F:=t\cdot[|\nabla f|^2-\alpha\cdot  \partial_tf]_+\in L^\infty(B_{3R/2,T}).$
Put
  $$M_1:=\sup_{B_{R,T} }F\qquad{\rm and}\qquad M_2:=\sup_{B_{3R/2,T} }F.$$
We can assume $M_1>0$. If not, we are done.

Now let us choose   $\phi(x)=\phi(r(x))$ to be a function of the distance $r$ to the fixed point $x_0$ with the following property that
$$  \frac{M_1}{2M_2}\ls \phi \ls1\ \ {\rm on}\ \ B_{3R/2},\qquad \phi =1\ \ {\rm on}\ \ B_R,\qquad \phi = \frac{M_1}{2M_2}\ \ {\rm on}\ \ B_{3R/2}\backslash B_{5R/4},$$
and
$$ - \frac{C}{R}\phi ^{\frac 1 2}\ls  \phi '(r)  \ls 0\quad {\rm and}\quad  |\phi ''(r)|   \ls \frac{C}{R^2} \qquad \forall\ r\in(0,3R/2)$$
for some  universal constant $C$ (which is independent of  $N,K,R$). Then we have
\begin{equation}\label{eq5.10}
\frac{|\nabla \phi |^2}{\phi }= \frac{|\phi '|^2|\nabla r|^2}{\phi }\ls  \frac{C^2}{R^2}:=\frac{C_1}{R^2} \quad\ {\rm on}\ \    B_{3R/2},
\end{equation}
and, by the Laplacian comparison theorem \cite[Corollary 5.15]{gig15} for $RCD^*(-K,N)$ with $N>1$ and $K>0$,   that
\begin{equation}\label{eq5.11}
\begin{split}
 \La\phi =\phi' \La r+\phi''|\nabla r|^2&\gs-\frac{C}{R} \Big(\sqrt{(N-1)K}\coth\big(r\sqrt{\frac{K}{N-1}}\big)\Big)  - \frac{C}{R^2} \\
 &   \gs -\frac{C}{R} \Big(\sqrt{(N-1)K }+\frac{N-1}{R}\Big)  - \frac{C}{R^2} \gs-C_2(\frac{\sqrt K}{ R}+\frac 1{R^2})
 \end{split}
 \end{equation}
on $  B_{3R/2}$, in the sense of distributions, where we have used that
$$\coth\big(r\sqrt{\frac{K}{N-1}}\big)\ls \coth\big(R\sqrt{\frac{K}{N-1}}\big)\ls  1+\frac{1}{R\sqrt{K/(N-1)}}. $$
We claim that the estimate  (\ref{eq5.11}) still holds for  $RCD^*(-K,N)$ with $N\gs1$ and $K\gs0$. Indeed,
in the case when $K=0$ and $N>1$, the Laplacian comparison theorem states $\La r\ls (N-1)/r$. Then (\ref{eq5.11}) still holds. In the case when $N=1$,
since that $(X,d,\mu)$ satisfies  $RCD^*(-K,N)$ implies that it satisfies  $RCD^*(-K,N+1)$, we can use the Laplacian comparison theorem for $RCD^*(-K,N+1)$ to conclude that (\ref{eq5.11}) still holds in this case. Therefore, the claim is proved.

 Here and in the sequel of this proof, we denote $C_{1},C_2,C_3,\cdots$  the various constants which  depend only on  $N$. (\ref{eq5.11}) implies that the distribution $\La \phi$ is a signed Radon measure (since
$\La \phi+C_2(\sqrt K/R+1/R^2)$ is a positive distribution). Then its absolutely continuous part $(\La \phi)^{\rm ac}\gs -C_2(\sqrt K/R+1/R^2)$ a.e. $x\in B_{3R/2}$ and its singular part $(\La \phi)^{\rm sing}\gs0.$

Put $G(x,t):= \phi F$. According to Lemma \ref{lem5.8}  and the Lebiniz rule \ref{lem3.2}(ii), we have  $G\in H^{1}(B_{3R/2,T})$  and, for almost every $t\in(0,T)$,   that the function $G(\cdot, t)$ satisfies that
$$\La G=F\La \phi+\phi\La F+2\ip{\nabla \phi}{\nabla F}$$
in the sense of distributions.  Fix arbitrarily a such $t\in(0,T)$. Then $\La G$ is a signed Radon measure on $B_{3R/2}$ with
\begin{equation}\label{eq5.12}
(\La G)^{\rm sing}=F(\La \phi)^{\rm sing}+\phi(\La F)^{\rm sing}\gs0
\end{equation}
and
$(\La G)^{\rm ac}=F(\La \phi)^{\rm ac}+\phi(\La F)^{\rm ac}+2\ip{\nabla \phi}{\nabla F}$ a.e. $x\in B_{3R/2}$. We have, for almost all $x\in B_{3R/2}$,
\begin{equation} \label{eq5.13}
\begin{split}
(\La G)^{\rm ac}-\partial_t G  +2\ip{\nabla f}{\nabla G} =&  \phi\Big((\La F)^{\rm ac} -\partial_t F  +2\ip{\nabla f}{\nabla F}\Big)\\
&\ +  F(\La \phi)^{\rm ac}+2\ip{\nabla \phi}{\nabla F}  +2\ip{\nabla f}{\nabla\phi}F.
\end{split}
\end{equation}
By (\ref{eq5.5}) and $G= \phi F$,  we have, for almost all $x\in  B_{3R/2}$,   that, for any fixed $\epsilon>0$,
  \begin{equation}\label{eq5.14}
\begin{split}
{\rm RHS \ of}\  (\ref{eq5.13})\overset{(\ref{eq5.5})}{\gs}&  \phi\Big[-\frac{F}{t}+2t\Big(\frac{1}{N}\big(|\nabla  f|^2-\partial_tf\big)^2-K|\nabla f|^2\Big)\Big] \\
&\ +G\frac{(\La \phi)^{\rm ac}}{\phi}+2\ip{\nabla \phi}{\nabla (G/\phi)}  +2\ip{\nabla f}{\nabla\phi}\frac{G}{\phi} \\
\gs\ & -\frac{G}{t}+2t\phi\Big[  \frac{1}{N}\big(|\nabla  f|^2-\partial_tf\big)^2-K|\nabla f|^2 \Big]\\
&\ +\frac{G}{\phi}\Big[-C_2\big(\frac{\sqrt K}{R}+\frac{1}{R^2}\big)-\frac{2C_1}{R^2}\Big] +2 \ip{\nabla \phi}{\nabla  G}/\phi  -2|\nabla f|\frac{|\nabla\phi|}{\phi}\cdot G\\
\gs\ & -\frac{G}{t}+2t\phi\Big[  \frac{1}{N}\big(|\nabla  f|^2-\partial_tf\big)^2-K|\nabla f|^2 \Big]\\
&\ -C_3\frac{G}{\phi} \big(\frac{\sqrt K}{R}+\frac{1}{R^2}\big)   +2 \ip{\nabla \phi}{\nabla  G}/\phi  -\epsilon \frac{G^2}{\phi}\frac{C_1}{R^2}-|\nabla f|^2\frac{1}{\epsilon},
\end{split}
\end{equation}
 where we have used (\ref{eq5.10}), (\ref{eq5.11}) and that, for any $\epsilon>0$, the following
$$2|\nabla f|\cdot G\frac{|\nabla \phi|}{\phi}\ls \epsilon G^2\frac{|\nabla\phi|^2}{\phi^2}+|\nabla f|^2\frac{1}{\epsilon}\ls  \epsilon \frac{G^2}{\phi}\cdot\frac{C_1}{R^2}+|\nabla f|^2\frac{1}{\epsilon}. $$

If we put
$$v=\frac{|\nabla f|^2}{F}$$
then we get $|\nabla f|^2=F\cdot v$ and
$$ F=t(|\nabla f|^2-\alpha\cdot\partial_tf)=t(F\cdot v-\alpha\cdot\partial_tf).$$
So
$$ \partial_tf=\frac{F(vt-1)}{\alpha t}.$$
Therefore we obtain
\begin{equation}\label{eq5.15}
\begin{split}
 -\frac{G}{t}+&2t \phi\Big[\frac{1}{N}\big(|\nabla f|^2-\partial_tf)^2-K|\nabla f|^2\Big]-\epsilon^{-1}|\nabla f|^2\\
 &= -\frac{G}{t}+ \phi\frac{2F^2}{N\alpha^2t}\Big((\alpha-1)vt+1\Big)^2-2tK \phi vF-\epsilon^{-1}v F\\
&\gs  -\frac{G}{t \phi }+\frac{2G^2}{N\alpha^2t \phi }\Big((\alpha-1)vt+1\Big)^2-\frac{2tKvG}{ \phi} -\epsilon^{-1}v \frac{G}{ \phi},
\end{split}
\end{equation}
where we have used that $0< \phi\ls 1$  and $KvG\gs0$. Denoting by
$$z:=(\alpha-1)vt\qquad {\rm and}\qquad A_\epsilon:=\frac{2Kt+\epsilon^{-1}}{\alpha-1},$$
we have
\begin{equation*}
 {\rm RHS\ of}\ (\ref{eq5.15})=\frac{1}{ \phi}\cdot\bigg(\frac{2G^2}{N\alpha^2t}\big(z+1\big)^2-\frac{G}{t}\big(1+A_\epsilon z\big)\bigg).
\end{equation*}
Finally $z\gs0$ implies that
$$\frac{1+A_\epsilon z}{(1+z)^2}\ls \max\Big\{1,\frac{1}{2}+\frac{A_\epsilon}{4}\Big\}\ls \max\Big\{1,\frac{1}{2}+\frac{Kt}{2(\alpha-1)}\Big\}+\frac{\epsilon^{-1}}{4(\alpha-1)}.$$
Denote by $$B_0:=\max\Big\{1,\frac{1}{2}+\frac{KT}{2(\alpha-1)}\Big\},$$
we have $\frac{1+A_\epsilon z}{(1+z)^2}\ls B_0+\frac{\epsilon^{-1}}{4(\alpha-1)},$ (since $K\gs0$ and $t\ls T$)
so \begin{equation*}
 {\rm RHS\ of}\ (\ref{eq5.15})\gs \frac{1}{ \phi}\cdot \frac{G}{t}\cdot\Big(\frac{2G}{N\alpha^2}-B_0-\frac{\epsilon^{-1}}{4(\alpha-1)}\Big)\cdot\big(z+1\big)^2 .
\end{equation*}
By combining this with (\ref{eq5.13}), (\ref{eq5.14}) and (\ref{eq5.15}), we obtain that
\begin{equation}\label{eq5.16}
\begin{split}
&(\La G)^{\rm ac} -\partial_t G  +2\ip{\nabla f}{\nabla G} -2 \ip{\nabla \phi}{\nabla  G}/\phi\\
 &\qquad\gs \frac{1}{ \phi}\cdot \frac{G}{t}\cdot\Big(\frac{2G}{N\alpha^2}-B_0-\frac{\epsilon^{-1}}{4(\alpha-1)}\Big)\cdot\big(z+1\big)^2\ -C_3\frac{G}{\phi} \big(\frac{\sqrt K}{R}+\frac{1}{R^2}\big)  -\epsilon \frac{G^2}{\phi}\frac{C_1}{R^2}.
\end{split}
\end{equation}
From the definition of $\phi$ and $F/t\in L^\infty(B_{3R/2,T})$ (by Lemma \ref{lem5.8}), we see that $G$ achieves one of its strict maximum in $B_{3R/2,T}$ in the sense of Theorem \ref{max-parabolic}.
By (\ref{eq5.12}), we know that $\La^{\rm sing}G\gs0$. Notice also $\partial_tf\in L^\infty(B_{2R,T})$ since $u\gs \delta>0$ and $\partial_tu\in H^1(B_{2R,T})\cap L^\infty(B_{2R,T})$
Hence, by using  Theorem \ref{max-parabolic} with  $w:=2f-2\ln\phi\in   H^{1}(B_{3R/2,T})\cap L^\infty(B_{3R/2,T})$, and combining with (\ref{eq5.16}), we conclude that there exit a sequence $\{x_j,t_j\}_{j\in\mathbb N}$ such that, for each $j\in\mathbb N$,
\begin{equation}\label{equation5.17}
G_j:=G(x_j,t_j)\gs \sup_{B_{3R/2,T}}G-1/j
\end{equation}
and that
\begin{equation}\label{equation5.18}
\begin{split}
 \frac{G_j}{t_j}&\cdot\Big(\frac{2G_j}{N\alpha^2}-B_0-\frac{\epsilon^{-1}}{4(\alpha-1)}\Big)\cdot\big(z(x_j,t_j)+1\big)^2 -C_3G_j\cdot \big(\frac{\sqrt K}{R}+\frac{1}{R^2}\big) -\epsilon G_j^2\cdot\frac{C_1}{R^2}\\
& \ls  \phi(x_j,t_j)\cdot\frac{1}{j}\ls \frac{1}{j}.
\end{split}
\end{equation}
We consider firstly the case when
$$\bar G:=\sup_{B_{3R/2,T}} G>\frac{N\alpha^2}{2}\Big(B_0+\frac{\epsilon^{-1}}{4(\alpha-1)}\Big).$$
In this case, the equation (\ref{equation5.17}) tells us $G_j\gs \frac{N\alpha^2}{2}\big(B_0+\frac{\epsilon^{-1}}{4(\alpha-1)}\big)$ for all sufficiently large $j$. Thus, from
(\ref{equation5.18}), we have
$$ \frac{G_j}{t_j}\cdot\Big(\frac{2G_j}{N\alpha^2}-B_0-\frac{\epsilon^{-1}}{4(\alpha-1)}\Big) -C_3G_j\cdot \big(\frac{\sqrt K}{R}+\frac{1}{R^2}\big) -\epsilon G_j^2\cdot\frac{C_1}{R^2} \ls \frac{1}{j}.$$
Letting $j\to\infty$, we have
$$ \frac{\bar G}{T}\cdot\Big(\frac{2\bar G}{N\alpha^2}-B_0-\frac{\epsilon^{-1}}{4(\alpha-1)}\Big)\ls C_3\bar G\cdot \big(\frac{\sqrt K}{R}+\frac{1}{R^2}\big) +\epsilon \bar G^2\cdot\frac{C_1}{R^2},$$
where we have used $t_j\ls T$ for all $j\in\mathbb N$. Thus, we have
\begin{equation}\label{equation5.19}
\bar G\ls \frac{B_0+\frac{\epsilon^{-1}}{4(\alpha-1)}+ C_3T\cdot \big(\frac{\sqrt K}{R}+\frac{1}{R^2}\big)}{ \frac{2}{N\alpha^2}-  \epsilon T\cdot\frac{C_1}{R^2}}.
\end{equation}
In the case when $\bar G\ls \frac{N\alpha^2}{2}\big(B_0+\frac{\epsilon^{-1}}{4(\alpha-1)}\big)$, it is clear that (\ref{equation5.19}) still holds.

Fix any $\beta\in(0,1)$. By choosing $\epsilon=2\beta R^2/(C_1\cdot N\alpha^2T)$. Then we conclude, by (\ref{equation5.19}), that
\begin{equation}\label{equation5.20}
\begin{split}
\bar G&\ls \frac{B_0+\frac{C_1\cdot N\alpha^2\cdot T}{8(\alpha-1)\cdot\beta R^2}+ C_3T\cdot \big(\frac{\sqrt K}{R}+\frac{1}{R^2}\big)}{ \frac{2}{N\alpha^2}(1-\beta)}\\
&=B_0\cdot \frac{N\alpha^2}{2}\cdot \frac{1}{1-\beta}+
\Big(\frac{C_1\cdot N^2\alpha^4 \cdot T}{16(\alpha-1)\cdot\beta R^2}+ C_3T\cdot \big(\frac{\sqrt K}{R}+\frac{1}{R^2}\big)\cdot\frac{N\alpha^2}{2}\Big)\cdot \frac{1}{1-\beta}\\
&\ls B_0\cdot \frac{N\alpha^2}{2}\cdot\frac{1}{1-\beta}+
\frac{C_4\cdot  \alpha^4\cdot T}{ (\alpha-1)R^2}\cdot\frac{1}{(1-\beta)\cdot\beta}+ C_5T\cdot \big(\frac{\sqrt K}{R}+\frac{1}{R^2}\big)\cdot\frac{\alpha^2}{1-\beta}.
\end{split}
\end{equation}
Therefore, we have
\begin{equation*}
\begin{split}
 \sup_{B_R\times(\gamma\cdot T,T]}F&\ls \sup_{B_{R,T}  }F \ls \sup_{B_{3R/2,T}}G\\
 &  \ls B_0\cdot \frac{N\alpha^2}{2}\cdot\frac{1}{1-\beta}+
\frac{C_4\cdot  \alpha^4\cdot T}{ (\alpha-1)R^2}\cdot\frac{1}{(1-\beta)\cdot\beta}+ C_5T\cdot \big(\frac{\sqrt K}{R}+\frac{1}{R^2}\big)\cdot\frac{\alpha^2}{1-\beta}.
\end{split}
\end{equation*}
By recalling
$F=t(|\nabla f|^2-\alpha\cdot  \partial_tf)_+$ and $B_0=\max\big\{1,\frac{1}{2}+\frac{KT}{2(\alpha-1)}\big\}$,
we conclude that the  local gradient estimate (\ref{eq5.9}) holds, since $t>\gamma\cdot T.$
This completes the proof.
 \end{proof}

Now, let us remove the additional assumption  $\partial_tu\in  H^{1}(B_{2R,T})\cap L^{\infty}(B_{2R,T})$ and  prove   Theorem \ref{thm1.4}.
\begin{proof}[Proof of Theorem \ref{thm1.4}] Let $\alpha>1$ and $\beta\in(0,1)$. Without loss of generality, we can assume that $T_*<\infty.$
Given any  $\delta>0$, from  \cite[Theorem 2.2]{stu96}, we have $u+\delta\in L_{\rm loc}^\infty(B_{2R,T_*}).$ Without loss the generality, we can assume that $u+\delta\in L^\infty(B_{2R,T_*}),$ since the desired result is a local estimate.

Given any $\varepsilon>0$, according to Lemma \ref{lem5.4} and Lemma \ref{lem5.6}, we can use Lemma \ref{lem5.9} to the Steklov averages $(u+\delta)_h$. Then, by an approximating argument (and  taking $\gamma=1-\beta$), we have
\begin{equation*}
 \begin{split}
\sup_{B_R\times((1-\beta)T,T] }\Big(\frac{|\nabla u|^2}{(u+\delta)^2}-\alpha\cdot\frac{\partial_tu}{u+\delta}\Big)(x,t) \ls& \max\bigg\{1, \frac{1}{2}+\frac{KT}{2(\alpha-1)}\bigg\} \cdot \frac{N\alpha^2}{2T} \cdot\frac{1}{(1-\beta)^2}\\
  & +
\frac{C_N\cdot  \alpha^4}{R^2 (\alpha-1)}\cdot\frac{1}{(1-\beta)^2\beta }+ \big(\frac{\sqrt K}{R}+\frac{1}{R^2}\big)\cdot\frac{C_N\cdot\alpha^2}{(1-\beta)^2}.
\end{split}
 \end{equation*}
  Letting $\delta (\in \mathbb Q)$ tend to $0^+$ and replacing $1-\beta$ by $\beta$, we have the desired (\ref{eq1.6}). By combining with the arbitrariness of $\varepsilon$, we complete the proof of   Theorem \ref{thm1.4}.
\end{proof}

\section{A sharp local Yau's gradient estimate}

Let $K\gs0 $, $N\in(1,\infty)$ and let $(X,d,\mu)$ be a metric measure space satisfying $RCD^*(-K,N)$.  Suppose that  $\Omega$  is a domain in $X$.
In this section, we will prove a sharp local Yau's gradient estimate---Theorem \ref{thm1.6}.

\begin{proof}[Proof of Theorem \ref{thm1.6}] Fix $\beta\in(0,1)$.
Let $u$ be a positive harmonic function on $B_{2R}:=B_{2R}(p)$  and let $f=\log u$. Without loss of generality, we can assume that $u\gs\delta$ for some $\delta>0$.  By  the chain rule \ref{lem3.2}(ii),   a direct computation shows that
$$\La f=-|\nabla f|^2\quad {\rm on }\quad B_{2R}.$$
Since $|\nabla f|\in L^\infty_{\rm loc}(B_{2R})$, by setting $g:=|\nabla f|^2$ and using  Corollary \ref{lem6.1},  (noticing that $N>1$)   we know that
$g\in H^1(B_{3R/2})\cap L^\infty(B_{3R/2}) $ and
$\La^{\rm sing}g\gs0$ and,
 for  $\mu$-a.e. $  x\in \big\{y:\ g(y) \not=0\big\}\cap B_{3R/2}$,
\begin{equation} \label{eq6.2}
\begin{split}
\frac{1}{2} \La^{\rm ac}g\gs& \frac{g^2}{N}- \ip{\nabla g}{\nabla f}-Kg+\frac{N}{N-1}\cdot\Big(\frac{\ip{ \nabla f}{\nabla g}}{2g}+\frac{g}{N}\Big)^2\\
 =& \frac{g^2}{N}- \ip{\nabla g}{\nabla f}-Kg+\frac{N}{N-1}\cdot\bigg[\Big(\frac{\ip{ \nabla f}{\nabla g}}{2g}\Big)^2+\frac{2\ip{ \nabla f}{\nabla g}}{2g}\cdot \frac{g}{N}+\Big(\frac{g}{N}\Big)^2\bigg]\\
\gs & \frac{g^2}{N-1}- \frac{N-2}{N-1}\cdot\ip{\nabla g}{\nabla f}-Kg.
\end{split}
\end{equation}

Since $g\in L^\infty(B_{3R/2})$, we define
$$M_1:=\sup_{B_R}g\quad {\rm and}\quad  M_2:=\sup_{B_{3R/2}}g.$$
We assume that $M_1>0$ (otherwise, we are done). Now let us choose   $\phi(x)=\phi(r(x))$ as above. That is, $\phi(x)$ is
 a function of the distance $r$ to the fixed point $x_0$ with the following property that
$$  \frac{M_1}{2M_2}\ls \phi \ls1\ \ {\rm on}\ \ B_{3R/2},\qquad \phi =1\ \ {\rm on}\ \ B_R,\qquad \phi = \frac{M_1}{2M_2}\ \ {\rm on}\ \ B_{3R/2}\backslash B_{5R/4},$$
and
$$ - \frac{C}{R}\phi ^{\frac 1 2}\ls  \phi '(r)  \ls 0\quad {\rm and}\quad  |\phi ''(r)|   \ls \frac{C}{R^2} \qquad \forall\ r\in(0,3R/2)$$
for some universal constant $C$ (which is independent of  $N,K,R$). Then we have, from (\ref{eq5.10})-(\ref{eq5.11}), that
\begin{equation} \label{equation6.2}
\frac{|\nabla \phi |^2}{\phi }\ls \frac{C_1}{R^2}\qquad {\rm and}\qquad \La\phi \gs-C_2(\frac{\sqrt K}{ R}+\frac 1{R^2})
\end{equation}
on $  B_{3R/2}$.  Then the distribution $\La \phi$ is a signed Radon measure and its absolutely continuous part $(\La \phi)^{\rm ac}\gs -C_2(\sqrt K/R+1/R^2)$ a.e. $x\in B_{3R/2},$ and its singular part $(\La \phi)^{\rm sing}\gs0.$ Here and in the sequel of this proof, we denote $C_{1},C_2,C_3,\cdots$  the various constants which  depend only on  $N$.

Put $G(x):= \phi\cdot g$.
According to  the Lebiniz rule \ref{lem3.2}(ii), we have  $G\in H^{1}(B_{3R/2})$  and
$$\La G=g\La \phi+\phi\La g+2\ip{\nabla \phi}{\nabla g}$$
in the sense of distributions. Then, by $\La^{\rm sing}g\gs0$ and $\La^{\rm sing}\phi\gs0$, we get $\La^{\rm sing}G\gs0$. The combination of (\ref{eq6.2}) and (\ref{equation6.2}) implies that
\begin{equation}\label{equation6.3}
\begin{split}
\La^{\rm ac}G \gs &\ {\phi}\La^{\rm ac}g+2\ip{\nabla \phi}{\nabla (G/\phi)}  + G\frac{(\La \phi)^{\rm ac}}{\phi} \\
\gs\ & 2\phi\Big( \frac{g^2}{N-1}- \frac{N-2}{N-1}\cdot\ip{\nabla g}{\nabla f}-Kg\Big)\\
& +2 \ip{\nabla \phi}{\nabla  G}/\phi+\frac{G}{\phi}\Big[-C_2\big(\frac{\sqrt K}{R}+\frac{1}{R^2}\big)-\frac{2C_1}{R^2}\Big] \\
\gs\ &  \frac{2}{\phi}\cdot \frac{G^2}{N-1}-\frac{2(N-2)}{N-1}\cdot\Big(\ip{\nabla G}{\nabla f}-G\ip{\nabla \phi}{\nabla f}/\phi\Big)-2KG \\
& +2 \ip{\nabla \phi}{\nabla  G}/\phi-C_3\cdot\frac{G}{\phi}  \big(\frac{\sqrt K}{R}+\frac{1}{R^2}\big)  \\
\gs\ &  \frac{2}{\phi}\cdot \frac{G^2}{N-1}-\frac{2(N-2)}{N-1}\cdot \ip{\nabla G}{\nabla f}-\frac{2(N-2)}{N-1}\cdot\Big( \epsilon \frac{G^2}{\phi}\cdot\frac{C_1}{R^2}+\frac{G}{\phi}\frac{1}{\epsilon}\Big)-2K\frac{G}{\phi} \\
& +2 \ip{\nabla \phi}{\nabla  G}/\phi-C_3\cdot\frac{G}{\phi}  \big(\frac{\sqrt K}{R}+\frac{1}{R^2}\big)
\end{split}
\end{equation}
 for any $\epsilon>0$, where we have used  $g=|\nabla f|^2=G/\phi$, $2KG\ls 2KG/\phi$ and that, for any $\epsilon>0$, the following
$$-G\ip{\nabla \phi}{\nabla f}/\phi\ls 2|\nabla f|\cdot G\frac{|\nabla \phi|}{\phi}\ls \epsilon G^2\frac{|\nabla\phi|^2}{\phi^2}+|\nabla f|^2\frac{1}{\epsilon}\ls  \epsilon \frac{G^2}{\phi}\cdot\frac{C_1}{R^2}+|\nabla f|^2\frac{1}{\epsilon}. $$
From the definition of $\phi$, we know that $G$ achieves one of its strict maximum in $B_{3R/2}$ in the sense of Theorem \ref{max-elliptic}.
Notice   that $\La^{\rm sing}G\gs0$.
Hence, according to  Theorem \ref{max-elliptic} for $w:=2\frac{N-2}{N-1}f-2\ln\phi\in   H^{1}(B_{3R/2})\cap L^\infty(B_{3R/2})$ (since $u\gs \delta>0$), and by combining with (\ref{equation6.3}), we conclude that there exit  a sequence $\{x_j\}_{j\in\mathbb N}$ such that, for each $j\in\mathbb N$,
\begin{equation}\label{equation6.4}
G_j:=G(x_j)\gs \sup_{B_{3R/2}}G-1/j
\end{equation}
and that (noticing that $\phi\in(0,1]$)
\begin{equation}\label{equation6.5}
\begin{split}
 2& \frac{G_j^2}{N-1}-\frac{2(N-2)}{N-1}\cdot\Big( \epsilon  G_j^2 \cdot\frac{C_1}{R^2}+G_j\frac{1}{\epsilon}\Big)-2KG_j
  -C_3\cdot G_j  \big(\frac{\sqrt K}{R}+\frac{1}{R^2}\big)\\
  &\quad\ls \phi(x_j)\cdot\frac{1}{j}\ls \frac{1}{j}
\end{split}
\end{equation}
 for any $\epsilon>0$. Letting $j\to\infty$ and denoting $\bar G:=\sup_{B_{3R/2}}G=\lim_j G_j$, we obtain
\begin{equation}\label{equation6.6}
\Big(\frac{1}{N-1}-\frac{(N-2)\epsilon\cdot C_1}{(N-1)R^2}\Big)\cdot \bar G\ls K+\frac{N-2}{(N-1)\epsilon}+\frac{C_3}{2}\big(\frac{\sqrt K}{R}+\frac{1}{R^2}\big)
\end{equation}
 for any $\epsilon>0$.

In the case when $N>2$, by  choosing $\epsilon=\frac{\beta\cdot R^2}{(N-2)\cdot C_1}$, we obtain from (\ref{equation6.6}) that
 \begin{equation*}
 \begin{split}
  \frac{1-\beta}{N-1}\cdot \bar G&\ls  K+ \frac{C_1\cdot (N-2)^2}{\beta R^2}+\frac{C_3}{2}\big(\frac{\sqrt K}{R}+\frac{1}{R^2}\big)\\
 &\ls K+ \frac{C_1\cdot (N-2)^2}{\beta R^2}+\beta K+\frac{C_3^2}{16\beta R^2}+\frac{C_3}{2R^2},
 \end{split}
 \end{equation*}
where we have used
$$ \frac{C_3}{2} \frac{\sqrt K}{R}=2\sqrt K\cdot \frac{C_3}{4R} \ls \beta K+\frac{1}{\beta} \frac{C_3^2}{(4R)^2}.$$
Then, we get
 \begin{equation}\label{equation6.7}
 \begin{split}
  \bar G&\ls  \frac{1+\beta}{1-\beta}(N-1) K  + \frac{N-1}{1-\beta}\cdot\frac{1}{\beta R^2}\Big( C_1\cdot (N-2)^2+\frac{C_3^2}{16}+\frac{C_3\beta}{2}\Big)\\
  &\ls  \frac{1+\beta}{1-\beta}(N-1) K  + \frac{C_4}{\beta(1-\beta)\cdot R^2},
 \end{split}
 \end{equation}
where we have used $\beta<1$.

In the case when $N\in (1,2]$,  from (\ref{equation6.6}), we have
 $$\frac{1}{N-1}\cdot \bar G \ls  K+\frac{C_3}{2}\big(\frac{\sqrt K}{R}+\frac{1}{R^2}\big)
 \ls K+\beta K+\frac{C_3^2}{16\beta R^2}+\frac{C_3}{2R^2}.$$
Thus, the estimate (\ref{equation6.7}) still holds in this case.

Therefore, the equation (\ref{equation6.7}) shows that, for any $\beta\in (0,1)$,
 \begin{equation*}
\sup_{B_R}g \ls  \frac{1+\beta}{1-\beta}(N-1) K  + \frac{C_4}{\beta(1-\beta)\cdot R^2}.
 \end{equation*}
 Now the proof is finished.
\end{proof}


\begin{thebibliography}{99}
\bibitem{ags15} L. Ambrosio, N. Gigli, G. Savar\'e, \emph{Bakry--Emery curvature-dimension condition and Riemannian Ricci curvature bounds}, Ann. Probab., 43 (2015), 339--404.
\bibitem{ags13-lip} L. Ambrosio, N. Gigli, G. Savar\'e, \emph{Density of Lipschitz functions and equivalence of weak gradients in metric measure
spaces}, Rev. Mat. Iberoam., 29 (2013), 969--996.

\bibitem{ags-duke} L. Ambrosio, N. Gigli, G. Savar\'e, \emph{Metric meausure spaces with Riemannian Ricci curvauture bounded from below}, Duke Math. J., 163(7) (2014), 1405--1490.

\bibitem{ags14} L. Ambrosio, N. Gigli, G. Savar\'e, \emph{Calculus and heat flow in metric measure spaces and applications to spaces with Ricci bounds from below}, Invent. Math., 195(2) (2014), 289--391.

\bibitem{ams15} L. Ambrosio, A. Mondino, G. Savar\'e, \emph{On the Bakry-\'Emery condition, the gradient estimates and the local-to global property of $RCD^*(K,N)$ metric measure spaces},  J. Geom. Anal.,  26(1) (2016), 24--56.

%\bibitem{akp11} S. Alexander, V. Kapovitch, A. Petrunin, \emph{Alexandrov geometry}, preprint.

\bibitem{bs10} K. Bacher \& K. Sturm, \emph{Localization and tensonrization properties of the curvature-dimension for metric measure spaces,}  J. Funct. Anal., 259(1) (2010) 28--56.
%\bibitem{bbi01} D. Burago,  Y. Burago, S. Ivanov, \emph{A Course in Metric Geometry}, Graduate Studies in Mathematics, vol. 33, AMS (2001).

\bibitem{bbg14}D. Bakry, F. Bolley \& I. Gentil, \emph{The Li-Yau inequality and applications under a curvature-dimension condition}, available at http://arxiv.org/abs/1412.5165.

\bibitem{bq00} D. Bakry, Z. Qian,  \emph{Some new results on eigenvectors via dimension, diameter, and Ricci curvature}, Adv. Math.  155(1)  (2000), 98--153.

%\bibitem{ber08} J. Bertrand, \emph{Existence and uniqueness of optimal maps on Alexandrov spaces}, Adv. in Math. 219 (2008), 838--851, MR2442054, Zbl 1149.49002.

%\bibitem{bm06} A. Bj\"orn, N. Marola, \emph{Moser iteration for (quasi)minimizers on metric spaces}, Manuscripta Math. 121, (2006), 339--366. %, MR2267657, Zbl 1123.49032.

%\bibitem{bjo02} J. Bj\"orn, \emph{Boundary continuity for quasiminimizers on metric spaces}, Illinois J. Math., 46(2) (2002), 383--403, MR1936925, Zbl 1026.49029.

%\bibitem{bgp92} Y. Burago, M. Gromov, G. Perelman, \emph{A. D. Alexandrov spaces with curvatures bounded below}, Russian Math. Surveys 47 (1992), 1--58, MR1185284, Zbl 0802.53018.

\bibitem{bl06}  D. Bakry \& M. Ledoux, \emph{A logarithmic Sobolev form of the Li-Yau parabolic inequality}, Rev. Mat. Iberoam., 22(2): 683--702, 2006.

\bibitem{bhllmy15} F. Bauer, P. Horn, Y. Lin, G. Lippner, D. Mangoubi \& S-T. Yau, \emph{Li-Yau inequality on graphs},  J. Differ. Geom., 99 (2015), 359-- 405.


\bibitem{bp04} H. Brezis \& A. C. Ponce, \emph{Kato's inequality when $\Delta u$ is a measure}, C. R. Acad. Sci. Paris, Ser. I 338, 599--604, 2004.

\bibitem{bq99}  D. Bakry \& Z. Qian, \emph{Harnack inequalities on a manifold with positive or negative Ricci curvature}. Rev. Mat. Iberoam., 15(1): 143--179, 1999.

%\bibitem{caf89} L. Caffarelli, \emph{Interior a priori estimates for solutions of fully nonlinear equations}, Ann. of Math. 130 (1989), 189--213, MR1005611, Zbl 0692.35017.


\bibitem{che99} J. Cheeger, \emph{Differentiability of Lipschitz functions on metric measure spaces}. Geom. Funct. Anal. 9,  (1999), 428--517.% MR1708448, Zbl 1200.58007.

%\bibitem{ccm95} J. Cheeger, T. Colding, W. P. Minicozzi II, \emph{Linear growth harmonic functions on complete manifolds with nonnegative Ricci %curvature}, Geom. Funct. Anal. 5(6) (1995), 948--954, MR1361516, Zbl 0871.53032.

\bibitem{cy75} S. Y.  Cheng, S. T. Yau, \emph{Differential equations on Riemannian manifolds and their geometric applications}, Comm. Pure Appl. Math. 28
(1975), 333--354.  % MR0385749, Zbl 0312.53031.

\bibitem{dav89}  E. B. Davies, \emph{Heat kernels and spectral theory}, volume 92 of Cambridge Tracts in
Mathematics. Cambridge University Press, Cambridge, 1989.

\bibitem{eks15} M. Erbar, K. Kuwada, K. Sturm, \emph{On the equivalence of the entropic curvature-dimension condition and Bochner¡¯s inequality on metric measure spaces}, Invent. Math., 201 (2015), 993--1071.

\bibitem{gm14} N. Garofalo, A .Mondino, \emph{Li-Yau and Harnack type inequalities in metric measure spaces}, Nonlinear Anal., 95(2014), 721--734.

\bibitem{gig15} N. Gigli, \emph{On the differential structure of metric measure spaces and applications}, Mem. Amer. Math. Soc. 236 (1113)
(2015).

\bibitem{gmo15}N. Giglia, S.Mosconi,   \emph{ The abstract Lewy-Stampacchia inequality and applications}, J. Math. Pures Appl. 104 (2) (2015) 258--275.

\bibitem{gm13} N. Gigli \& A. Mondino, \emph{A PDE approach to nonlinear potential theory}, J. Math. Pures Appl. 100 (4) (2013) 505--534.

%\bibitem{gko13} N. Gigli, K. Kuwada, S. Ohta, \emph{Heat flow on Alexandrov spaces}.  Comm. Pure Appl. Math. 66 (2013), 307--331.

%\bibitem{gt01} D. Gilbarg, N. Trudinger, \emph{Elliptic partial differential equations of second order}, Springer-Verlag,
%Berlin, 2001,  MR1814364, Zbl 1042.35002.

%\bibitem{hom89} L. H\"ormander, \emph{The analysis of linear partial differential operators I, 2th edition}, in `Grundlehren der mathematischen
%Wissenschaften',  256, Springer-Verlag, Berlin, 1989,  MR1996773, Zbl 1028.35001.

\bibitem{haj03} P. Haj{\l}asz, \emph{Sobolev spaces on metric-measure spaces}, Heat kernels and analysis on manifolds, graphs, and metric spaces (Paris, 2002),  173--218, Contemp. Math., 338, Amer. Math. Soc., Providence, RI, 2003.

\bibitem{hk00} P. Haj{\l}asz, P. Koskela, \emph{Sobolev met Poincar\'e}, Mem. Am. Math. Soc. 145(688), (2000), x--101. % MR1683160, Zbl 0954.46022.

\bibitem{hom89} L. H\"ormander, \emph{The analysis of linear partial differential operators I}, 2th edition,
in `Grundlehren der mathematischen Wissenschaften' 256, Springer-Verlag, Berlin, 1989.


\bibitem{hkx13} B. Hua, M. Kell \& C. Xia, \emph{Harmonic functions on metric measure spaces}, available at http://arxiv.org/abs/1308.3607.
%\bibitem{ham93}   R. Hamilton, \emph{A matrix Harnack estimate for the heat equation.} Comm. Anal. Geom., 1(1):88--99, 1993.

\bibitem{hx14} B. Hua, \& C. Xia,
\emph{A note on local gradient estimate on Alexandrov spaces}, Tohoku Math. J.   66(2)  (2014),  259--267.

\bibitem{jen88} R. Jensen, \emph{The maximum principle for viscosity solutions of fully nonlinear second order partial differential equations,} Arch. Rat. Mech. Anal.,  101 (1988) 1--27.

\bibitem{jia14} R. Jiang, \emph{Cheeger-harmonic functions in metric measure spaces revisited,} J. Funct. Anal.,  266 (2014) 1373--1394.

\bibitem{jia15} R. Jiang, \emph{The Li-Yau inequality and heat kernels on metric measure spaces,} J. Math. Pures Appl. 104 (9) (2015) 29--57.

\bibitem{jky14} R. Jiang, P. Koskela \& D. Yang, \emph{Isoperimetric inequality via Lipschitz regularity of Cheeger-harmonic functions},
J. Math. Pures Appl., 101(2014), 583--598.

\bibitem{jiang-z16} R. Jiang \& H. C. Zhang, \emph{Hamilton¡¯s gradient estimates and a monotonicity formula for heat flows on metric measure spaces,} Nonlinear Anal., 131(2016), 32--47.

\bibitem{jz16} Y. Jiang \& H. C. Zhang, \emph{Sharp spectral gaps  on metric measure spaces,} Calc. Var. PDE. 55 (2016),  no. 1, Art. 14, 14 pp.

\bibitem{lee12}  P. W. Y. Lee, \emph{Generalized Li-Yau estimates and Huisken's monotonicity formula}, available at http://arxiv.org/abs/1211.5559.

\bibitem{lx11}  J. Li \& X. Xu, \emph{Differential Harnack inequalities on Riemannian manifolds I: linear heat equation}. Adv. Math., 226(5): 4456--4491, 2011.

\bibitem{ly86}   P. Li \& S.-T. Yau,  \emph{On the parabolic kernel of the Schr\"odinger operator}. Acta Math., 156(3-4):153--201, 1986.

\bibitem{li05} X.-D. Li, \emph{Liouville theorems for symmetric diffusion operators on complete Riemannian manifolds}, J. Math. Pures Appl. 54, 1295--1361,  2005.

%\bibitem{kap02} V. Kapovitch, \emph{Regularity of limits of noncollapsing sequences of manifolds}, Geom. Funct. Anal., 12 (2002), 121--137, MR1904560, Zbl 1013.53046.


%\bibitem{km02} J. Kinnunen, O. Martio, \emph{Nonlinear potential theory on metric spaces}, Illinois J. Math., 46(3) (2002), 857--883, MR1951245, Zbl 1030.46040.

%\bibitem{km03} J. Kinnunen, O. Martio, \emph{Potential theory of quasiminimizers}. Ann. Acad. Sci. Fenn. Math., 28  (2003), 459--490.  MR1996447, Zbl 1035.31007.

%\bibitem{kshan01} J. Kinnunen, N. Shanmugalingam, \emph{Regularity of quasi-minimizers on metric spaces}. Manuscripta Math. 105 (2001), 401--423. %  MR1856619, Zbl 1006.49027.

%\bibitem{kms01} K. Kuwae, Y. Machigashira, T. Shioya, \emph{Sobolev spaces, Laplacian and heat kernel on Alexandrov spaces}, Math. Z. 238(2) (2001), 269--316.       %MR1865418, Zbl 1001.53017.

%\bibitem{ks93} M. Korevaar, R. Schoen, \emph{Sobolev spaces and harmonic maps for metric space targets}, Comm. Anal. Geom. 1 (1993), 561--659.      % MR1266480, Zbl 0862.58004.

%\bibitem{krs03} P. Koskela, K. Rajala, N. Shanmugalingam, \emph{Lipschitz continuity of
%Cheeger-harmonic functions in metric measure spaces}. J. Funct. Anal., 202(1) (2003), 147--173, MR1994768, Zbl 1027.31006.

%\bibitem{ks07} K. Kuwae \& T. Shioya, \emph{Laplacian comparison for Alexandrov spaces},  available at http://arxiv.org/abs/0709.0788.

%\bibitem{ks10} K. Kuwae \& T. Shioya, \emph{ Infinitesimal Bishop-Gromov condition for Alexandrov spaces}, Probabilistic Approach to Geometry. Adv Stud Prue Math 57. Tokyo: Math Soc Japan, 2010: 293--302

%\bibitem{ks01} K. Kuwae, T. Shioya, \emph{On generalized measure contraction property
%and energy functionals over Lipschitz maps}, ICPA98 (Hammamet), Potential Anal. 15(1-2) (2001), 105--121, MR1838897, Zbl 0996.31006.

%\bibitem{ks03} K. Kuwae, T. Shioya, \emph{Sobolev and Dirichlet spaces over maps
%between metric spaces}, J. Reine Angew. Math. 555 (2003), 39--75, MR1956594, Zbl 1053.46020.

%\bibitem{kuw08} K. Kuwae, \emph{Maximum principles for subharmonic functions via local semi--Dirichlet forms}, Canad. J. Math. 60 (2008), 822--874, %MR2432825, Zbl 1160.31007.


%\bibitem{li86} P. Li, \emph{Large time behavior of the heat equation on complete manifolds with non-negative Ricci curvature}, Ann. of Math. 124 (1986) %1--21, MR0847950, Zbl 0613.58032.

%\bibitem{lv00} Y. Y. Li, M. Vogelius, \emph{Gradient estimates for solutions to divergence form elliptic equations with discontinuous coefficients}, Arch. %Rational Mech. Anal. 153 (2000), 91--151, MR1770682, Zbl 0958.35060.

\bibitem{lv09} J. Lott, C. Villani,  \emph{Ricci curvature for metric-measure spaces via optimal transport}, Ann. of Math. 169 (2009), 903--991 % $ MR2480619, Zbl 1178.53038.

\bibitem{lv07} J. Lott, C. Villani,  \emph{Weak curvature bounds and functional inequalities}, J. Funct. Anal. 245(1) (2007), 311--333.
 %MR2311627, Zbl 1119.53028.

%\bibitem{lv07-hj} J. Lott, C. Villani,  \emph{Hamilton--Jacobi semigroup on length spaces and applications}, J. Math. Pures Appl. 88 (2007), 219--229, %MR2355455, Zbl 1210.53047.

\bibitem{lw02} P. Li,  J. Wang,  \emph{Complete manifolds with positive spectrum. II.}  J. Differential Geom.  62(1)  (2002),   143--162.

\bibitem{mm13} N. Marola, \& M. Masson, \emph{On the Harnack inequality for parabolic minimizers in metric measure spaces}, Tohoku Math. J., 65 (2013), 569--589.

\bibitem{mn14} A. Mondino, A. Naber, \emph{Structure theory of metric measure spaces with lower Ricci curvature bounds I}, available at http://arxiv.org/abs/1405.2222.

\bibitem{omo67} H. Omori, \emph{Isometric immersions of Riemannian manifolds}, J. Math. Soc. Japan 19 (1967) 205--214.

%\bibitem{oht07} S. Ohta,  \emph{On measure contraction property of metric measure spaces}, Comment. Math. Helvetici, 82(4)  (2007),  805--828.  % MR2341840, Zbl 1176.28016.

%\bibitem{ots95} Y. Otsu, \emph{Almost every existence of second differentiable structure of Alexandrov spaces}, Preprinted, 1995.

%\bibitem{os94} Y. Otsu, T. Shioya, \emph{The Riemannian structure of Alexandrov spaces}, J. Differ. Geom. 39 (1994), 629--658.   %MR1274133, Zbl 0808.53061.

%\bibitem{per94} G. Perelman,  \emph{Elements of Morse theory on Alexandrov spaces},  St. Petersburg Math. J., 5(1) (1994), 205--213, MR1220498, Zbl 0815.53072.

%\bibitem{per-dc} G. Perelman,   \emph{DC structure on Alexandrov spaces.} Preprint,
%preliminary version available online at www.math.psu.edu/petrunin/

%\bibitem{pp95} G. Perelman, A. Petrunin, \emph{Quasigeodesics and gradient curves in Alexandrov spaces}, Preprint, available online at
%www.math.psu.edu/petrunin/

%\bibitem{pet98} A. Petrunin, \emph{Parallel transportation for Alexandrov spaces with curvature bounded below}, Geom. Funct. Anal.  8(1) (1998), 123--148, %MR1601854, Zbl 0903.53045.

\bibitem{pet11}  A. Petrunin, \emph{Alexandrov meets Lott--Villani--Sturm},   M\"{u}nster J. of Math. 4 (2011), 53--64.

%\bibitem{pet07}  A. Petrunin, \emph{Semiconcave Functions in Alexandrov¡¯s Geometry,} Surveys in Differential Geometry XI: Metric and Comparison Geometry, (2007),  pp. 137--201, International Press, Somerville, MA. %MR2408266,

%\bibitem{pet96}  A. Petrunin, \emph{Subharmonic functions on Alexandrov space}, Preprint (1996),  available online at
%www.math.psu.edu/petrunin/

%\bibitem{pet03}  A. Petrunin, \emph{Harmonic functions on Alexandrov space and its
%applications,}
%ERA Amer. Math. Soc., 9 (2003), 135--141, MR2030174, Zbl 1071.53527.

\bibitem{qia14}   B. Qian, \emph{Remarks on differential Harnack inequalities}. J. Math. Anal. Appl., 409(1): 556--566, 2014.

\bibitem{qzz13} Z. Qian, H.-C. Zhang \& X.-P. Zhu, \emph{Sharp spectral gap and Li-Yau's estimate on Alexandrov spaces},  Math. Z., 273(3-4) (2013) 1175--1195.

\bibitem{raj12-cvpde} T. Rajala,  \emph{Local Poincar\'e inequalities from stable curvature conditions on metric spaces}, Calc. Var. PDE 44(3-4)  (2012), 477--494.

%\bibitem{sav14} G. Savar\'e, \emph{Self-improvement of the Bakry-\'Emery condition and Wasserstein contraction of the heat flow in $RCD(K,\infty)$ metric measure spaces}, Discrete Contin. Dyn. Syst. 34 (2014) 1641--1661.

\bibitem{shan00} N. Shanmugalingam, \emph{Newtonian spaces:An extension of Sobolev spaces to metric measure spaces}. Rev. Mat. Iberoam. 16,(2000), 243--279. % MR1809341, Zbl 0974.46038.


%\bibitem{sa02} L. Saloff-Coste, \emph{Aspects of Sololev--type inequalities}, Londan Mathematical Society Lecture Note Series. 289, Cambridge Univ. press, %2002.

\bibitem{sz06}  P. Souplet \& Q. S. Zhang, \emph{Sharp gradient estimate and Yau's Liouville theorem for
the heat equation on noncompact manifolds.} Bull. London Math. Soc., 38(6) (2006) 1045--1053.


\bibitem{stu06} K. Sturm,  \emph{On the geometry of metric measure spaces. I, II}. Acta Math.  196(1) (2006), 65--131, 133--177.%MR2237207, Zbl 1106.53032.

%\bibitem{s98} K. Sturm,  \emph{Diffusion processes and heat kernels on metric spaces}, Ann. Probab. 26 (1998), 1--55, MR1617040, Zbl 0936.60074.

\bibitem{stu95}K. Sturm, \emph{Analysis on local Dirichlet spaces. II. Upper Gaussian estimates for the fundamental solutions of parabolic equations}, Osaka J. Math., 32(2)  (1995),   275--312.

\bibitem{stu96} K. Sturm, \emph{Analysis on local Dirichlet spaces. III. The parabolic Harnack inequality},
J. Math. Pures Appl., 75 (1996), 273--297. % MR1387522, Zbl 0854.35016.

%\bibitem{sy94} R. Schoen,  S. T. Yau, \emph{Lectures on Differential Geometry},
%International Press, Boston, 1994, MR1333601, Zbl 0830.53001.

%\bibitem{v09} C. Villani,  \emph{Optimal transport, old and new}. Grundlehren der mathematischen
%Wissenschaften, Vol. 338, Springer, 2008, MR2459454, Zbl 1156.53003.

\bibitem{y75} S. T. Yau, \emph{Harmonic functions on complete Riemannian manifolds}. Comm. Pure Appl. Math. 28 (1975), 201--228.
% MR0431040, Zbl 0291.31002.

\bibitem{zz10} H. C. Zhang, X. P.  Zhu, \emph{Ricci curvature on Alexandrov spaces and rigidity theorems}, Comm. Anal. Geom. 18(3) (2010), 503--554.
% MR2747437, Zbl 1230.53064.

%\bibitem{zz10-2} H. C. Zhang, X. P.  Zhu, \emph{On a new definition of Ricci curvature on Alexandrov spaces}, Acta Mathematica Scientia. 30B(6) (2010), %1949--1974. % MR2778704, Zbl pre05952627.

\bibitem{zz12} H. C. Zhang, X. P. Zhu, \emph{Yau's gradient estimates on Alexandrov spaces}, J. Differ. Geom., 91(3) (2012), 445--522.

\bibitem{zz14} H. C. Zhang, X. P. Zhu, \emph{Lipschitz continuity of harmonic maps between Alexandrov spaces}, available at http://arxiv.org/abs/1311.1331.
\end{thebibliography}
\end{document}